\documentclass[a4paper]{amsart}
\usepackage{color,stmaryrd}
\usepackage[dvipsnames]{xcolor}
\usepackage{amssymb}
\usepackage{amscd}
\usepackage{amsthm}
\usepackage{amsmath}
\usepackage{latexsym}
\usepackage[all]{xy}
\usepackage{mathrsfs}

\usepackage[czech,english]{babel}
\usepackage{hyperref}

\theoremstyle{plain}
\newtheorem{theorem}{Theorem}
\newtheorem{corollary}[theorem]{Corollary}
\newtheorem{lemma}[theorem]{Lemma}
\newtheorem{proposition}[theorem]{Proposition}
\newtheorem*{maintheorem}{Main Theorem}

\theoremstyle{definition}
\newtheorem{definition}[theorem]{Definition}
\newtheorem{example}[theorem]{Example}

\newtheorem{problem}[theorem]{Problem}

\theoremstyle{remark}

\newtheorem{remark}[theorem]{Remark}

\numberwithin{theorem}{section}
\numberwithin{equation}{section}

\makeatletter
\renewcommand{\p@enumii}{}
\makeatother

\DeclareMathOperator{\Add}{Add}

\DeclareMathOperator{\Gen}{Gen}
\DeclareMathOperator{\Cogen}{Cogen}
\DeclareMathOperator{\End}{End}
\DeclareMathOperator{\im}{Im}

\newcommand{\Spec}[1]{{\mathrm{Spec}\,#1}}

\newcommand{\ProjA}{{\mathrm{Proj}\textrm{-}A}}
\newcommand{\projA}{{\mathrm{proj}\textrm{-}A}}

\newcommand{\MorA}{{\mathrm{Mor}\textrm{-}A}}
\newcommand{\AMor}{{A\textrm{-}\mathrm{Mor}}}
\newcommand{\MorpA}{{\mathrm{Mor_{proj}}\textrm{-}A}}

\newcommand{\Hom}[3]{{\mathrm{Hom}_{#1}(#2,#3)}}
\newcommand{\Rhom}[3]{{\mathbf{R}\mathrm{Hom}_{#1}(#2,#3)}}
\newcommand{\lten}[1]{{\otimes^\mathbf{L}_#1}}
\newcommand{\Ext}[4]{{\mathrm{Ext}^{#1}_{#2}(#3,#4)}}
\newcommand{\Tor}[4]{{\mathrm{Tor}_{#1}^{#2}(#3,#4)}}

\newcommand{\rfmod}[1]{{\mathrm{mod}\textrm{-}{#1}}}
\newcommand{\rmod}[1]{{\mathrm{Mod}\textrm{-}{#1}}}

\newcommand{\lmod}[1]{{{#1}\textrm{-}\mathrm{Mod}}}

\newcommand{\ModA}{\rmod{A}}
\newcommand{\AMod}{\lmod{A}}
\newcommand{\modA}{\rfmod{A}}
\newcommand{\ModB}{\rmod{B}}
\newcommand{\BMod}{\lmod{B}}

\newcommand{\Ab}{\mathrm{Ab}}
\newcommand{\DerA}{\mathsf{D}(A)}

\newcommand{\Ann}[2]{\mbox{\rm{Ann}}_{#1}(#2)}
\newcommand{\Supp}[1]{\mbox{\rm{Supp}}\,#1}

\newcommand{\Ker}[1]{\mbox{\rm{Ker}}\,#1}
\newcommand{\Coker}[1]{\mbox{\rm{Coker}}\,#1}

\newcommand{\la}{\longrightarrow}
\newcommand{\dd}{\colon}
\newcommand{\id}{\mathrm{id}}
\newcommand{\op}{\mathrm{op}}
\newcommand{\torp}[1]{{\top\!}_{#1}}

\newcommand{\p}{\mathfrak{p}}
\newcommand{\q}{\mathfrak{q}}

\newcommand{\m}{\mathfrak{m}}

\newcommand{\Acal}{\ensuremath{\mathcal{A}}}
\newcommand{\Bcal}{\ensuremath{\mathcal{B}}}
\newcommand{\Ccal}{\ensuremath{\mathcal{C}}}
\newcommand{\Dcal}{\ensuremath{\mathcal{D}}}

\newcommand{\Fcal}{\ensuremath{\mathcal{F}}}
\newcommand{\Gcal}{\ensuremath{\mathcal{G}}}

\newcommand{\Lcal}{\ensuremath{\mathcal{L}}}

\newcommand{\Scal}{\ensuremath{\mathcal{S}}}
\newcommand{\Tcal}{\ensuremath{\mathcal{T}}}
\newcommand{\Ucal}{\ensuremath{\mathcal{U}}}
\newcommand{\Vcal}{\ensuremath{\mathcal{V}}}
\newcommand{\Xcal}{\ensuremath{\mathcal{X}}}
\newcommand{\Ycal}{\ensuremath{\mathcal{Y}}}

\newcommand{\Z}{\mathbb{Z}}
\newcommand{\QQ}{\mathbb{Q}}

\renewcommand{\iff}{if and only if }

\begin{document}

\title[Flat and silting epimorphisms]%
{Flat epimorphisms and silting epimorphisms coincide for commutative rings}

\author{Jan {\v S}{\v{t}}ov{\'{\i}}{\v{c}}ek}

\address{Charles University, Faculty of Mathematics and Physics, Department of Algebra, Sokolovsk\'a 83, 186 75 Praha, Czech Republic}
\email{stovicek@karlin.mff.cuni.cz}

\thanks{The research was supported by the grant GA~\v{C}R 23-05148S from the Czech Science Foundation.}

\date{October 2, 2025}

\begin{abstract}
We investigate the relation between partial silting modules, Gabriel topologies, and ring epimorphisms, with a particular emphasis on commutative rings. We show that a ring epimorphism of commutative rings is flat if and only if it is a silting ring epimorphism.
\end{abstract}

\maketitle

\setcounter{tocdepth}{2}
\tableofcontents

\bigskip

\section{Introduction} \label{sec:intro}

In this paper, we demonstrate a surprisingly tight connection between two types of homomorphisms of commutative rings of very different origins. We stress that no additional hypotheses, such as being noetherian or coherent, are necessary.

\begin{maintheorem} [see Theorem~\ref{thm:flat is silting}]
Let $A$ be a commutative ring and $\lambda\colon A \longrightarrow B$ be a ring homomorphism. Then $\lambda$ is a flat epimorphism if and only if $\lambda$ is a silting epimorphism (and $B$ is automatically commutative in this case).
\end{maintheorem}

The former of the two notions, flat ring epimorphism, is classical; see~\cite[Chapitre~IV]{Laz} and the references there. A homomorphism of rings $\lambda\colon A \longrightarrow B$ is a \emph{flat epimorphism} (see Definition~\ref{def:flat-epi}) if it is an epimorphism in the category of rings and $B$ is flat as a right $A$-module. If $A$ is commutative, so is $B$; and the most well known representatives of flat epimorphisms are localizations $A\la S^{-1}A$ at multiplicative sets. But there are more, as already noted in~\cite[\S5]{Laz}.

From a geometric point of view, the map $\lambda^\flat\colon\Spec{B}\la\Spec{A}$ induced by a flat epimorphism of commutative rings $\lambda$ is a homeomorphism onto its image by~\cite[Corollaire~IV.2.2]{Laz}, and $\lambda$ is determined up to equivalence by $\im\lambda^\flat$ by~\cite[Proposition~IV.2.5]{Laz}. More precisely, flat epiclasses of a commutative ring $A$ correspond to those subsets $U\subseteq\Spec(A)$ which become affine schemes with the sheaf induced from $\Spec{A}$.

The question of which subsets $U\subseteq\Spec(A)$ precisely induce flat epimorphisms from $A$ is subtle. Such a subset $U$ has to be closed under generalization by~\cite[Proposition~IV.2.5]{Laz} and, as explained in Theorems~\ref{thm:perfectGT and ringepi} and~\ref{thm:Gabrielcorr} below, the complement $V=\Spec{A}\setminus U$ even has to be a Thomason subset (i.~e.\ $U$ has to be an intersection of quasi-compact Zariski open subsets of $\Spec{A}$, a condition relevant only if $A$ is non-noetherian).
If $A$ is noetherian, each minimal prime in $V=\Spec{A}\setminus U$ is of height $\le 1$ by~\cite[Theorem~4.9]{AMSTV}, but all these conditions are still not sufficient; see~\cite[Example~4.13 and~\S6.1]{AMSTV}.
For any commutative $A$, the collection of all subsets of $\Spec{A}$ that induce flat epimorphisms is known to be closed under arbitrary intersections; essentially, by~\cite[Lemme~1.0']{Laz}.

\smallskip

However, our main aim here is not to give a classification of flat epimorphisms from a commutative ring $A$, but rather to prove that they all arise as a generalized localization of $A$ in a suitable sense. Namely, if $\Sigma=\{\sigma_i\colon P_i\la Q_i\}$ is a collection of maps between projective $A$-modules, we will look for ring homomorpshims $\lambda\colon A\la B$ universal with respect to the property that $\sigma\otimes_AB$ is an isomorphism for each $\sigma\in\Sigma$. If $S\subseteq A$ is a multiplicative subset, the classical localization $A\la S^{-1}A$ enjoys this property for $\Sigma=\{A\overset{s\cdot-}\la A\mid s\in S\}$. If $\Sigma$ is a set of maps between finitely generated projective modules, the universal localization in the sense of Schofield~\cite[Chap.~4]{Sch} satisfies the universal property and it is known to be flat by~\cite[Corollary~4.4]{AMSTV}, but not all flat epimorphisms (even for $A$ commutative noetherian) are universal localizations; see~\cite[\S6.4]{MS} or~\cite[Proposition~6.13]{AMSTV}. So in general we must allow $\Sigma$ to contain certain maps of infinitely generated projective modules.

This brings us to silting ring epimorphisms. Tilting modules and later tilting complexes were invented to compare categories of modules over different finite-dimensional algebras~\cite{BBtilt,HaRi}, culminating in a powerful derived Morita theory~\cite{RickardDM,Kel-DG}. However, from our perspective, a more relevant fact is the relation of tilting modules to torsion pairs and, in view of~\cite[\S1.2]{HRS}, to $t$-structures. In that context, tilting complexes were generalized to silting complexes to capture a larger class of $t$-structures~\cite{KV-aisles}. Returning from derived categories to module categories over finite-dimensional algebras, we arrive at the concept of support $\tau$-tilting module in the sense of~\cite[\S2]{AIR}, which is simply a module theoretic incarnation of a silting complex, as explained in~\cite[\S3]{AIR}. A subsequent work~\cite{AMV1} allowed us to leave the territory of finite-dimensional algebras and introduced a general definition of a silting module, generalizing support $\tau$-tilting modules from~\cite{AIR} to any rings. All in all, the main aim was still to have a tool to be able to encode a nice enough torsion pair in a module category by a single silting module. Specializing back to commutative algebra, such torsion pairs are precisely those given in terms of supports in Thomason sets \cite{AH}.

One might wonder what the previous paragraph has to do with ring theoretic localizations. On the one hand, it was noticed that examples of tilting modules can be constructed directly from classical or universal localizations~\cite{GL,AHT1,AS1,AA}, which was important, for example, for the classification result \cite{AS2}. On the other hand, localization with respect to a partial tilting or silting module (i.~e.\ a suitable summand of a tilting or silting module) appeared explicitly in~\cite{CTT}, and in the context of finite-dimensional algebras also in~\cite{Jas} and~\cite[\S4]{MS-tor}. These results culminated in~\cite[Theorem~3.7]{AMV2}, which identified intervals in the lattice of torsion pairs in $\ModA$ with lattices of all torsion pairs in $\ModB$ for suitable ring epimorphisms $\lambda\colon A\la B$. Such epimorphisms $\lambda$ are constructed from projective presentations $\sigma\colon P\la Q$ of partial silting modules and, as we shall see, are universal with respect to making $\sigma$ invertible. These are the \emph{silting epimorphisms} from the main theorem.

Although silting epimorphisms $\lambda\colon A\la B$ of non-commutative rings are usually not flat, they always satisfy $\Tor{1}{A}{B}{B}=0$. If $A$ is commutative noetherian, this is in fact equivalent to flatness of $B$ by~\cite[Proposition~4.5]{AMSTV}. Although this criterion does not work for commutative non-noetherian rings, we prove that silting epimorphisms are still flat for them (Theorem~\ref{thm:flat is silting}). More importantly, we prove in Theorem~\ref{thm:from flat to silting} that all flat epimorphisms of commutative rings are silting, and a combination of these two results yields our main theorem.

\smallskip

To summarize, we have the following hierarchy of equivalence classes of ring epimorphisms originating from a commutative ring $A$ (see Remark~\ref{rem:silting epi}(2) for the second inclusion).
\[
\left\{
\begin{matrix}
\textrm{localizations at}\\
\textrm{multiplicative sets}\\
\lambda\colon A \longrightarrow B
\end{matrix}
\right\}
\subseteq
\left\{
\begin{matrix}
\textrm{universal}\\
\textrm{localizations}\\
\lambda\colon A \longrightarrow B
\end{matrix}
\right\}
\subseteq
\left\{
\begin{matrix}
\textrm{silting}\\
\textrm{epimorphisms}\\
\lambda\colon A \longrightarrow B
\end{matrix}
\right\}
=
\left\{
\begin{matrix}
\textrm{flat}\\
\textrm{epimorphisms}\\
\lambda\colon A \longrightarrow B
\end{matrix}
\right\}.
\]
In general, the first two inclusions are proper by~\cite[Theorem 5.13]{AMSTV} and~\cite[\S6.4]{MS}.

\medskip
\noindent\textbf{Notation.}
Throughout the paper, $A$ stands for a unital ring, $\ModA$ the category of right $A$-modules, and $\AMod$ for the category of left $A$-modules. 
We will denote by $\modA$ the subcategory of $\ModA$ formed by finitely presented modules, and by $\ProjA$ and $\projA$ the subcategories of all projective $A$-modules and of all finitely generated projective right $A$-modules, respectively.
All subcategories are assumed to be full and strict.
Furthermore, we write $\DerA$ for the unbounded derived category of $\ModA$. Given a module $M\in\ModA$,  we denote by $\Add{M}$ the class of all modules which are isomorphic to direct summands of direct sums of copies of $M$, and by $\Gen{M}$ the class of all $M$-generated modules, i.~e.~all epimorphic images of modules in $\Add M$. 
 
Given a subcategory $\Ccal\subseteq\ModA$
and a set of non-negative integers $I$ (which is usually expressed by symbols such as $\geq n$, $\leq n$, or just $n$, with the obvious meaning),
we will use the following notation:
\begin{align*}
\Ccal^{\perp_I}&=\{X\in\ModA \mid \Ext{i}{R}{\Ccal}{X}=0 \text{ for all } i\in I\}, \\
{^{\perp_I}\Ccal}&=\{X\in\ModA \mid \Ext{i}{R}{X}{\Ccal}=0\text{ for all } i\in I\}, \\
\Ccal^{\torp{I}}&=\{Y\in\AMod \mid \Tor{i}{R}{\Ccal}{Y}=0 \text{ for all } i\in I\}. \\
\end{align*}
If $\Ccal$ consists of a single module $M$, we simply write $M^{\perp_I}$, ${^{\perp_I}M}$ or $M^{\torp{I}}$, respectively.

\subsection*{Acknowledgment}
I would like to thank Leonid Positselski for helpful discussions about Gabriel filters.


\section{Preliminaries} \label{sec:prelim}

\subsection{Torsion pairs and silting modules} \label{subsec:torsion and silting}
A pair $(\Tcal,\Fcal)$ of full subcategories of an abelian category $\Acal$ is said to be a \emph{torsion pair} if $\Hom{\Acal}{\Tcal}{\Fcal}=0$ and for every $M$ in $\Acal$ there are $T$ in $\Tcal$, $F$ in $\Fcal$ and a short exact sequence
$$\xymatrix{0\ar[r]&T\ar[r]^f&M\ar[r]^g&F\ar[r]&0}.$$
For a torsion pair $(\Tcal,\Fcal)$, we say that $\Tcal$ is a \emph{torsion class} and $\Fcal$ is a \emph{torsion-free class}. If $\Acal$ admits arbitrary (set-indexed) products and coproducts, it is easy to see that $\Tcal$ is closed under coproducts, extensions and epimorphic images, while $\Fcal$ is closed under products, extensions and subobjects. If $\Acal$ is also well powered (that is, if the class of subobjects of any given object forms a set), then these closure conditions characterize torsion classes and torsion-free classes (see \cite{Di}).
This is the case in particular if $\Acal=\ModA$.
We further say that a torsion pair $(\Tcal,\Fcal)$ is \emph{hereditary} if $\Tcal$ is closed under subobjects, and we say that it is of \emph{finite type} if $\Fcal$ is closed under direct limits.

Given a ring $A$, we will also consider the category $\MorA$ of morphisms of right $A$-modules. Objects in this category are precisely morphisms $\gamma\colon M_1\longrightarrow M_0$ of $A$-modules, and the morphisms are given by commutative squares of the form
\[ \xymatrix{M_1\ar[d]\ar[r]^\gamma & M_0\ar[d]\\N_1\ar[r]^{\delta} & N_0.} \]
We will denote the object of $\MorA$ corresponding to a morphism $\gamma\colon M_1\longrightarrow M_0$ in $\ModA$ by $C_\gamma$. Thus, the commutative square above corresponds to a morphism $C_\gamma\la C_\delta$ in $\MorA$.
Similarly, we will denote by $\AMor$ the category of morphisms of the left $A$-modules and use an analogous notation for objects there.
 
It is well known that $\MorA$ is equivalent to the category of right modules $\rmod{T_2(A)}$ over the $2\times 2$ upper triangular matrix ring $T_2(A)$ over $A$; see e.~g.\ \cite[Section~III.2]{ARS}. Analogously, $\AMor$ is equivalent to $\lmod{T_2(A)}$. Transporting the structure of the tensor product over $T_2(A)$ through these equivalences, we obtain a bifunctor
\begin{equation}\label{eq:matrix-tensor}
\otimes_{T_2(A)}\dd \MorA\times\AMor\la\Ab.
\end{equation}
By tracing the equivalences, we can give an explicit formula for this tensor product. If $\gamma\dd M_1\la M_0$ is a morphism in $\ModA$ and $\beta\dd N^0\la N^1$ is a morphism in $\AMod$, then $C_\gamma\otimes_{T_2(A)}C_\beta$ is given by the pushout of abelian groups:
\begin{equation}\label{eq:tensor maps}
\vcenter{\xymatrix{
M_1\otimes_AN^0 \ar[r]^{\gamma\otimes_AN^0} \ar[d]_{M_1\otimes_A\beta} & M_0\otimes_AN^0 \ar[d] \\
M_1\otimes_AN^1 \ar[r] & C_\gamma\otimes_{T_2(A)}C_\beta.
}}
\end{equation}
We will also use the corresponding notation for the left derived functors 
\[
\operatorname{Tor}^{T_2(A)}_{i}\dd \MorA\times\AMor\la\Ab.
\]
and given a set $I$ of non-negative integers and $\Ccal\subseteq\MorA$, we will use the notation for orthogonal classes $\Ccal^{\torp{I}}\subseteq\AMor$ in this context.

The subcategory of $\MorA$ whose objects are morphisms in $\ProjA$, i.~e.\ morphisms between projective $A$-modules, will be denoted by $\MorpA$.
There is an obvious full functor
\[ \MorpA \la \DerA \]
that sends $C_\gamma = (M_1\overset{\gamma}\la M_0)$ to a two-term complex concentrated in homological degrees $1$ and $0$ (which is also the mapping cone of $\gamma$). Abusing notation, we will also denote this complex by $C_\gamma\in\DerA$.
Given a morphism $\sigma$ in $\ProjA$, we will consider the class of \emph{$\sigma$-divisible modules} (see~\cite[\S4.3]{MS}) 
\begin{align*}
\Dcal_\sigma&=\{X\in \ModA \mid \Hom{A}{\sigma}{X}\ \text{is surjective} \} \\
&=\{X\in \ModA \,\mid\, \Hom{\DerA}{C_\sigma}{X[1]}=0 \}.
\end{align*}
If $\Sigma$ is a set of morphisms in $\ProjA$, we set $\Dcal_\Sigma=\bigcap_{\sigma\in\Sigma} \Dcal_\sigma$ and call the modules in this class \emph{$\Sigma$-divisible}. Dually, we will consider the class of \emph{$\sigma$-torsion-free} left modules
\[
\Fcal_\sigma=\{Y\in \AMod \mid \sigma\otimes_AY\ \text{is injective} \}
\]
denote $\Fcal_\Sigma=\bigcap_{\sigma\in\Sigma} \Fcal_\sigma$ and call the modules in this class \emph{$\Sigma$-torsion-free}.
A word of warning is due here: $\Fcal_\Sigma$ is not always a torsion-free class, conditions for this are discussed in the rest of this subsection.

We need the following easy observations about the category $\MorA$ and the classes of divisible and torsion-free modules just defined.

\smallskip
\begin{proposition}\label{prop:morphismcat}
Let $A$ be a ring.
\begin{enumerate}
\item The following are equivalent for an object $C_\gamma$ in $\MorA$:
\begin{enumerate}
\item $C_\gamma$ is projective in $\MorA$;
\item $\gamma$ is in $\Add(\{{\rm id}:P\longrightarrow P \mid P\in\ProjA \} \cup \{ 0\longrightarrow P\mid P\in\ProjA \})$
\item $\gamma$ is a split monomorphism of projective $A$-modules.
\end{enumerate} 
In particular, projective objects are contained in the subcategory $\MorpA$ of $\MorA$.

\item Given a morphism $\sigma\dd P_1\la P_0$ in $\ProjA$, then
\[ C_\sigma^{\perp_1} = \{C_\gamma\in\MorA\,\mid\,\Coker\gamma\in\Dcal_\sigma\}\subseteq\MorA. \]

\item Given a morphism $\sigma\dd P_1\la P_0$ in $\ProjA$, then
\[ C_\sigma^{\torp1} = \{C_\beta\in\AMor\,\mid\,\Ker\beta\in\Fcal_\sigma\}\subseteq\AMor. \]
\end{enumerate}
\end{proposition}

\begin{proof}
We refer to \cite[Propostion 2.5]{MV} for (1) and to \cite[Lemma 5.4]{MS} for (2). 

(3) The character dual $(-)^+\dd\Hom{\Z}{-}{\QQ}$ induces an obvious functor $\MorA^\op\la\AMor$ which is compatible with the equivalences with modules over triangular matrix rings. Given $C_\beta\in\AMor$, this gives us a well-known isormophism $\big(\Tor{1}{T_2(A)}{C_\sigma}{C_\beta}\big)^+\cong\Ext{1}{\AMor}{C_\beta}{C_{\sigma^+}}$. In particular, $C_\sigma^{\torp1}={^{\perp_1}(C_{\sigma^+})}$ and, as $\sigma^+$ is a morphism of injective left $A$-modules, a dual analog version of \cite[Lemma 5.4]{MS} shows that $C_\beta\in{^{\perp_1}(C_{\sigma^+})}$ if and only if $\Hom{A}{\Ker\beta}{\sigma^+}\cong(\sigma\otimes_A\Ker\beta)^+$ is surjective. The latter condition is equivalent to $\sigma\otimes_A\Ker\beta$ being injective.
\end{proof}

Classes of modules of the form $\Dcal_\Sigma$ are always closed under extensions, quotients, and products, see~\cite[Lemma~3.6(1)]{AMV1}. Dually, classes of the form $\Fcal_\Sigma$ are always closed under extensions, submodules and coproducts. So $\Dcal_\Sigma$ is a torsion class if and only if it is closed under direct sums, and $\Fcal_\Sigma$ is a torsion-free class if and only if it happens to be closed under products.

\begin{definition}[{\cite[Definitions~2.1 and~3.7]{AMV1}}]\label{def:p silting}
We say that an $A$-module $T$ is
\begin{itemize}
\item  \emph{partial silting} if it admits a projective presentation $ P\stackrel{\sigma}{\longrightarrow} Q\longrightarrow T\longrightarrow 0$ such that 
$\mathcal{D}_\sigma$ is a torsion class containing $T$;
\item \emph{silting} if it admits  a projective presentation $ P\stackrel{\sigma}{\longrightarrow} Q\longrightarrow T\longrightarrow 0$ such that 
$\mathcal{D}_\sigma=\Gen{T}$,
\item \emph{tilting}  if $T^{\perp_1}=\Gen{T}$, or equivalently (cf.\ \cite[Proposition 3.13]{AMV1}), $T$ is silting with respect to a monomorphic projective presentation $\sigma$.
\end{itemize}
 The torsion class $\Gen{T}$ generated by a silting (respectively, tilting) module $T$ is called a \emph{silting} (respectively, \emph{tilting}) \emph{class}. Two silting modules $T$ and $T'$ are said to be \emph{equivalent} if they generate the same silting class, which amounts to having the same additive closure $\Add{T}=\Add{T^\prime}$ by discussion after~\cite[Proposition~3.10]{AMV1}.
\end{definition}

For a morphism $\sigma\dd P_1\la P_0$ between projective $A$-modules, we also consider the following subcategories of $\DerA$,
\begin{align*}
\Ucal_\sigma&=\{X\in\DerA\mid\Hom{\DerA}{C_\sigma}{X[>0]}=0\}, \\
\Vcal_\sigma&=\{X\in\DerA\mid\Hom{\DerA}{C_\sigma}{X[\leq0]}=0\}.
\end{align*}
We say that $C_\sigma$ is a (2-term) \emph{silting complex} if $(\Ucal_\sigma,\Vcal_\sigma)$ is a t-structure, that is, if $\Hom{\DerA}{\Ucal_\sigma}{\Vcal_\sigma}=0$ and every object $X$ of $\DerA$ fits into a triangle
\[ U\la X\la V\la U[1] \]
with $U$ in $\Ucal_\sigma$ and $V$ in $\Vcal_\sigma$. The next proposition follows from \cite[Lemma 6.2]{MS} and \cite[Theorem 4.9]{AMV1}. 

\begin{proposition}\label{prop:silting to tilting reduction}
The following are equivalent for an $A$-module $T$ with a projective presentation $\sigma$.
\begin{enumerate}
\item $T$ is silting with respect to $\sigma$;
\item $C_\sigma\oplus C_{\id_A}$ is tilting in $\MorA$;
\item $C_\sigma$ is a silting complex in $\DerA$.
\end{enumerate}
\end{proposition}

\begin{remark}\label{rem:silting proj pres}
As explained in~\cite[Remark~3.8(2)]{AMV1}, the class $\Dcal_\sigma$ in the definition of a partial silting module depends on the chosen projective presentation $\sigma$ and not just on $T$---it can happen that the same module $T$ is partial silting with respect to different projective presentations giving different torsion classes $\Dcal_\sigma$.

However, if $T$ is silting with respect to two projective presentations $\sigma$ and $\sigma'$, then necessarily $\Add C_\sigma=\Add C_{\sigma'}$ in $\DerA$.
This follows from the facts that
\begin{enumerate}
\item $\Ucal_\sigma=\Ucal_{\sigma'}=\{X\in\DerA\mid H_{<0}X=0\textrm{ and }H_0X\in\Gen T\}$ by~\cite[Theorem~4.9]{AMV1} and
\item $\Add C_\sigma=\{X\in\Ucal_\sigma\mid\Hom{\DerA}{X}{\Ucal_\sigma[1]}=0\}$ by \cite[Lemma~4.5]{AMV1}.
\end{enumerate}
\end{remark}

The following result says that each partial silting module can be completed to a silting module associated with the same torsion class. In particular, divisibility classes for projective presentations of partial silting modules are silting classes themselves.

\begin{theorem}[{\cite[Theorem 3.12]{AMV1}}]\label{thm:completion}
Let $T_1$ be a partial silting $A$-module with respect to a projective presentation $\sigma_1$. Then there exists a silting $A$-module $T$ with respect to a projective presentation $\sigma$ such that $C_{\sigma_1}$ is a direct summand of $C_\sigma$ in $\MorA$ and $\Dcal_{\sigma_1}=\Dcal_\sigma=\Gen T$.
\end{theorem}

Finally, we recall that even though silting classes are associated to a map between potentially large projectives, they are determined by a set of maps between finitely generated projective modules.

\begin{theorem}[{\cite[Theorem 2.3]{AH}, \cite[Theorem 6.3]{MS}}]\label{thm:finite type silting}
A torsion class $\Tcal\subseteq\ModA$ is silting (i.~e.\ it is of the form $\Tcal=\Gen{T}$ for some silting module $T$) if and only if it is of the form $\Tcal=\Dcal_\Sigma$ for a set $\Sigma$ of morphisms in $\projA$.
\end{theorem}

As a consequence, we have a criterion for the classes of the form $\Fcal_\sigma$ introduced above to be torsion-free classes.

\begin{corollary}\label{cor:Sigma-torsion-free-classes}
Suppose that $T\in\ModA$ is a partial silting module with respect to a projective presentation $\sigma$. Then the class $\Fcal_\sigma$ of the $\sigma$-torsion-free modules is a torsion-free class in $\AMod$.
\end{corollary}

\begin{proof}
We can without loss of generality assume that $T$ is silting by Theorem~\ref{thm:completion}. Let $\Sigma$ be a set of morphisms in $\projA$ given by Theorem~\ref{thm:finite type silting} such that $\Dcal_\sigma=\Dcal_\Sigma$. Then by Proposition~\ref{prop:silting to tilting reduction}, $T=C_\sigma\oplus C_{\id_A}$ is a tilting object in $\MorA$ and if we put $\Scal=\{C_{\tau}\mid\tau\in\Sigma\}$, then
\[ T^{\perp_1} = \Scal^{\perp_1} = \{C_\gamma\in\AMor\mid\Ker\beta\in\Dcal_\sigma\}. \]
It follows from~\cite[Theorem 15.2]{GT} that also $T^{\torp1}=\Scal^{\torp1}$ and one more application of Proposition~\ref{prop:silting to tilting reduction} gives $\Fcal_\sigma=\Fcal_\Sigma$.
Since the maps in $\Sigma$ are between finitely generated (thus also presented) projective modules, it follows that $\Fcal_\Sigma$ is closed under products.
\end{proof}

\subsection{Flat rings epimorphisms and Gabriel filters}
We say that a ring homomorphism $\lambda\colon A\longrightarrow B$ is a
\emph{ring epimorphism} if it is an epimorphism in the category of rings with unit, or equivalently by~\cite[Proposition~1.1]{Si} and~\cite[Proposition 1.2]{Stor}, or by~\cite[Lemma~1.1]{GdP}, if the functor given by the restriction of scalars $\lambda_\ast\colon\ModB\la \ModA$ is fully faithful. Two ring epimorphisms $\lambda_1\colon A\longrightarrow B_1$ and $\lambda_2\colon A\longrightarrow B_2$ are said to be \emph{equivalent} if there is a (necessarily unique) isomorphism of rings $\mu\colon B_1\longrightarrow B_2$ such that $\lambda_2=\mu\circ \lambda_1$. We then say that $\lambda_1$ and $\lambda_2$ are in the same {\em epiclass} of $A$. The subcategories obtained as the essential image of the restriction of scalars functor are of particular interest to us.

\begin{definition} 
A full subcategory $\Xcal$ of an abelian category $\Acal$ is called 
\begin{enumerate} 
\item \emph{wide} if it is closed under kernels, cokernels and extensions in $\Acal$. Equivalently, the inclusion $\Xcal \longrightarrow \Acal$ is an exact functor of abelian categories, and the canonical morphism $\Ext1\Xcal--\longrightarrow\Ext1\Acal--$ is an isomorphism.
\item \emph{bireflective} if the inclusion functor $\Xcal\longrightarrow\Acal$ admits both left and right adjoints.
\end{enumerate}
\end{definition}

\begin{theorem}\label{thm:epicl}
Let $A$ be a ring. A subcategory $\Xcal\subseteq\ModA$ is bireflective if and only if it is closed under limits and colimits (equivalently: under kernels, cokernels, products and copoducts).
The assignment that takes a ring epimorphism $\lambda\colon A\longrightarrow B$ originating in $A$ to the essential image $\Xcal_B$ of $\lambda_\ast\dd\ModB\la \ModA$ defines a bijection between
\begin{itemize}
\item epiclasses of ring epimorphisms $A\la B$, and
\item bireflective subcategories of $\ModA$,
\end{itemize}
which restricts to a bijection between  
\begin{itemize}
\item epiclasses of ring epimorphisms $A\la B$ with $\Tor{1}{A}{B}{B}=0$, and
\item wide bireflective subcategories in $\ModA$.
\end{itemize}
\end{theorem}

\begin{proof}
In~\cite[Theorem~1.2]{GdP} it was shown that the assignment $\lambda\mapsto\im\lambda_\ast$ induces a bijection between epiclasses of ring epimorphisms $\lambda\dd A\la B$ and subcategories of $\ModA$ closed under limits and colimits. Each such subcategory is automatically bireflective---given a ring epimorphism $\lambda\dd A\la B$, the induction $B\otimes_A-$ and the coinduction $\Hom{A}{B}{-}\dd\ModA\la\ModB$ are the adjoints of the fully faithful functor $\lambda_\ast\dd\ModB\la \ModA$.
Finally, the first bijection restricts itself to the second one by~\cite[Theorem~4.8]{Sch}.
\end{proof}

An important class of ring epimorphisms, which will be under focus in this paper and which was of particular interest also historically (cf.~\cite[Chap.~IV]{Laz} and~\cite{Si,Stor}), consists of those whose target is flat over the domain.

\begin{definition}\label{def:flat-epi} A {ring epimorphism} $\lambda\colon A\longrightarrow B$  is said to be
{\em (right) flat} if $B$ is a flat right $A$-module.
\end{definition}

The omnipresence of these epimorphisms, especially over commutative rings, is attested by the following examples.

\begin{example} The following are examples of flat ring epimorphisms.
\begin{itemize}
\item If $A$ is a commutative ring and $S\subseteq A$ is a multiplicative set, then the localization $A\longrightarrow S^{-1}A$ is a flat epimorphism. More generally, any universal localization of $A$ in the sense of Schofield~\cite[Chap.~4]{Sch} is a flat epimorphism by~\cite[Corollary 4.4]{AMSTV}.

\item If $A$ is not necessarily commutative and $S$ is a multiplicative set consisting of central elements, then $A\longrightarrow S^{-1}A$ is still a flat epimorphism. More generally, any Ore localization of A is flat (see e.~g.\ \cite[\S V.2, Proposition~5]{G-thesis}).
\end{itemize}
\end{example}

Flat ring epimorphisms correspond bijectively to the so-called perfect Gabriel filters \cite[\S V.2]{G-thesis}. Here we briefly recall the relevant terminology and refer to~\cite{St}
for details. 

\begin{definition}
A non-empty subset of the set of left ideals of $A$ is said to be a \emph{filter} if it is closed under intersections and, whenever $I$ lies in $\Gcal$, so does every $J$ containing $I$.

A filter $\mathcal{G}$ of left ideals of $A$ is said to be a \emph{(left) Gabriel filter}, if the following conditions hold:

\begin{enumerate}
\item[(i)] if $I \in \mathcal{G}$, then for any $x \in A$ the ideal $(I : x) = \{a \in A \mid ax \in I\}$ belongs to~$\mathcal{G}$,
\item[(ii)] if $J$ is a left ideal such that there is $I \in \mathcal{G}$ with $(J : x) \in \mathcal{G}$ for all $x \in I$, then $J \in \mathcal{G}$.
\end{enumerate}
\end{definition}

Gabriel filters are often referred to as \emph{Gabriel topologies} (see~\cite[\S V.2]{G-thesis} or~\cite[\S VI.5]{St}), since they form a system of neighborhoods of the zero element for a topology on $A$. By~\cite[Theorem~VI.5.1]{St}, Gabriel filters are in bijecton with hereditary torsion pairs $(\mathcal{T}_\mathcal{G},\mathcal{F}_\mathcal{G})$ in $\AMod$, which assigns to a Gabriel filter $\Gcal$ the classes
\begin{align*}
\mathcal{T}_\mathcal{G}&=\{M\in\AMod\mid \Ann{A}{m}\in\Gcal \text{ for all } I\in\Gcal\}, \\
\mathcal{F}_\mathcal{G}&=\{M\in\AMod\mid \Hom{A}{A/I}{M}=0 \text{ for all } I\in\Gcal\}.
\end{align*}
Here $\Ann{A}{m}=\{a\in A\mid am=0\}$ stands for the annihilator of $m\in M$. Modules in $\Tcal_\Gcal$ are called \emph{$\Gcal$-torsion} and modules in $\Fcal_\Gcal$ are said to be $\mathcal{G}$-\emph{torsion-free}. Recall that a torsion pair $(\Tcal,\Fcal)$ in $\AMod$ is hereditary if and only if $\Fcal$ is closed under taking injective envelopes, \cite[Proposition VI.3.2]{St}.

Gabriel filters are also in one-to-one correspondence with the so-called \emph{Giraud subcategories} of $\AMod$ in the sense of~\cite[\S X.1]{St}, that is, the full subcategories $\Xcal$ for which the inclusion functor $i\colon\mathcal X\hookrightarrow\lmod{A}$ has an exact left adjoint functor $\ell\colon \lmod{A}\longrightarrow \mathcal X$. This is~\cite[Theorem X.2.1]{St} and the Giraud subcategory corresponding to a Gabriel filter $\mathcal G$ is given as
\begin{align*}
\mathcal X_\Gcal
&=\{M\in \AMod\mid I\subseteq M \textrm{ induces }\Hom{A}{A}{M}\cong\Hom{A}{I}{M} \text{ for\ all } I\in\mathcal G\}
\\
&=\{M\in \AMod\mid \Hom{A}{A/I}{M}=0=\Ext{1}{A}{A/I}{M} \text{ for\ all } I\in\mathcal G\}
\\
&=\{M\in \AMod\mid \Hom{A}{T}{M}=0=\Ext{1}{A}{T}{M} \text{ for\ all } T\in\Tcal_\Gcal\}.
\end{align*}
The modules in $\Xcal_\Gcal$ are called \emph{$\mathcal G$-closed} modules, \cite[\S IX.1]{St}. 
The unit $\psi$ associated to the adjoint pair $(\ell,i)$ induces for every module $M\in\AMod$ an exact sequence
\begin{equation}\label{eq:reflect}
0\longrightarrow t(M)\longrightarrow M\stackrel{\psi_M}{\longrightarrow} M_\Gcal\longrightarrow \Coker\psi_M\longrightarrow 0 \end{equation}
where $t(M)$ and $\Coker\psi_M$ are in $\mathcal T_\Gcal$, and $M_\Gcal\in\Xcal_\Gcal$.
In fact, any map $\psi_M$ in an exact sequence with these properties is automatically an $\mathcal X_\Gcal$-reflection of $M$, i.~e.~every morphism $h\colon M\longrightarrow X$ with $X\in\Xcal_\Gcal$ factors uniquely through $\psi_M$.
All of this follows from the explicit construction of $\ell$ in~\cite[\S IX.1]{St} and the proof of~\cite[Proposition IX.1.11]{St}.
The $\mathcal X_\Gcal$-reflection of $A$ always induces a ring homomorphism $\lambda\colon A\longrightarrow A_\Gcal$
using the natural ring structure on $A_\Gcal\cong{\rm End}_A A_\Gcal$. 
Every module in $\mathcal X_\Gcal$ is an $A_\Gcal$-module,
so we have a left exact functor $q\colon \lmod{A}\stackrel{\ell}{\longrightarrow}\mathcal \Xcal_\Gcal\stackrel{j}{\hookrightarrow}\lmod{A_\Gcal}$.

\begin{definition}[{\cite[\S XI.3]{St}}]\label{def:perfect} A Gabriel filter $\mathcal G$ is said to be \emph{perfect} if the inclusion functor $i\colon\mathcal X_\Gcal\hookrightarrow\AMod$ has a right adjoint, or equivalently,  if $\mathcal X_\Gcal$ is closed under colimits.
 \end{definition} 

When $\Gcal$ is perfect, $\mathcal X_\Gcal$ is a bireflective subcategory of $\AMod$, and $\lambda\dd A\la A_\Gcal$ is the corresponding ring epimorphism under the bijection from Theorem~\ref{thm:epicl}. Since $\ell=A_\Gcal\otimes_A-$ is the (exact) left adjoint of the inclusion functor $i\colon\mathcal X\hookrightarrow\rmod{A}$, the ring epimorphism $\lambda$ is flat. The following theorem summarizes the discussion.
 
\begin{theorem}[{\cite[Ch.~XI, Theorem 2.1 and Proposition 3.4]{St}}]\label{thm:perfectGT and ringepi}
For any ring $A$, there are bijections between 
\begin{enumerate}
\item[(i)] perfect left Gabriel filters on $A$;
\item[(ii)] bireflective Giraud subcategories of $\AMod$;
\item[(iii)] epiclasses of right flat ring epimorphisms $A\longrightarrow B$.
\end{enumerate}
The bijection ${\rm (i)}\to{\rm (iii)}$ is defined by the assignment
$\Gcal\mapsto (\lambda\colon A\longrightarrow A_\Gcal)$, and the inverse map is given by
 $  (\lambda\colon A\longrightarrow B)\mapsto\,\mathcal G=\{ I\le A\,\mid\, B\cdot\lambda(I)=B\}$.
\end{theorem}

\subsection{Approximations and perpendicular categories}

Given any category $\Acal$ and a subcategory $\Bcal\subseteq\Acal$, we call a morphism $f\dd B\la A$ in $\Acal$ a \emph{$\Bcal$-precover} of the object $A$ if $B\in\Bcal$ and any morphism $g\dd B'\la A$ with $B'\in\Bcal$ factors through~$f$.
Dually, a morphism $f\dd A\la B$ is called a \emph{$\Bcal$-preenvelope} if $B\in\Bcal$ and each $g\dd A\la B'$ with $B'\in\Bcal$ factors through $f$.
The subcategory $\Bcal$ is called \emph{precovering} in $\Acal$ if each $A\in\Acal$ admits a $\Bcal$-precover and it is called \emph{preenveloping} if each $A\in\Acal$ admits a $\Bcal$-preenvelope.

Precovers and preenvelopes are generally collectively referred to as approximations.
There also exist stronger versions of these notions ($\Bcal$-covers, $\Bcal$-envelopes, covering and enveloping classes \cite[Definitions~5.1 and~5.5]{GT}), but these are not essential for the current discussion.

For us, important examples of precovering and preenveloping classes are given by the following proposition.

\begin{proposition}[{\cite[Theorems~6.11(a) and~6.25(a)]{GT}}] \label{prop:precovering and preenveloping}
Given any set of right modules $\Scal\subseteq\ModA$, the class $\Scal^{\perp_1}\subseteq\ModA$ is preenveloping.
Similarly, the class ${^{\torp1}\Scal}\subseteq\ModA$ is precovering for any set $\Scal\subseteq\AMod$ of left modules.
\end{proposition}

This also allows one to construct approximations (in fact, even reflections and coreflections) for other classes of modules.

\begin{lemma} \label{lem:perpendicular}
\begin{enumerate}
\item Let $\Scal\subseteq\ModA$ be a set of right modules of projective dimension at most one. Then $\Scal^{\perp_{0,1}}$ is a wide subcategory closed under products in $\ModA$.
\item Let $\Scal\subseteq\AMod$ be a set of left modules of flat dimension at most one. Then ${^{\torp{0,1}}\Scal}$ is a wide subcategory closed under direct sums in $\ModA$.
\end{enumerate}
\end{lemma}

\begin{proof}
The class $\Scal^{\perp_{0,1}}$ in~(1) is clearly closed under products and the fact that it is wide follows from \cite[Proposition 1.1]{GL}. 
The proof of~(2) is analogous.
\end{proof}

\begin{proposition}\label{prop:perp (co)relflection}
\begin{enumerate}
\item Let $\Scal\subseteq\ModA$ be a set of right modules of projective dimension at most one. Then the inclusion $\Scal^{\perp_{0,1}}\hookrightarrow\ModA$ admits a left adjoint. That is, for each $M\in\ModA$ there exists a homomorphism $\psi_M\dd M\la N$ with $N\in\Scal^{\perp_{0,1}}$ and such that each $\theta\dd M\la N'$ with $N'\in\Scal^{\perp_{0,1}}$ uniquely factors through~$\psi_M$.

\item Let $\Scal\subseteq\AMod$ be a set of left modules of flat dimension at most one. Then the inclusion ${^{\torp{0,1}}\Scal}\hookrightarrow\ModA$ admits a right adjoint. That is, for each $M\in\ModA$ there exists a homomorphism $\varphi_M\dd N\la M$ with $N\in{^{\torp{0,1}}\Scal}$ and such that each $\zeta\dd N'\la M$ with $N'\in{^{\torp{0,1}}\Scal}$ uniquely factors through~$\varphi_M$.
\end{enumerate}
\end{proposition}

\begin{proof}
We only prove (1) here; the argument for (2) is analogous.
We will denote $\Lcal=\Scal^{\perp_{0,1}}$ here for convenience.

Observe first that $\Lcal$ is a preenveloping class of modules. Indeed, any module $M$ admits an $\Scal^{\perp_1}$-preenvelope $M\la L$ by Proposition~\ref{prop:precovering and preenveloping}. As $\Scal^{\perp_0}\subseteq\ModA$ is a torsion-free class, we can factor out the corresponding torsion part $t_\Scal(L)$ from $L$. We clearly have $L/t_\Scal(L)\in\Scal^{\perp_0}$ and since $\Scal$ contains only modules of projective dimension at most one, and so $\Scal^{\perp_1}$ is closed under taking quotients, also $L/t_\Scal(L)\in\Scal^{\perp_1}$. It immediately follows that the composition $M\to L\twoheadrightarrow L/t_\Scal(L)$ is an $\Lcal$-preenvelope.

Since $\Lcal$ is closed under kernels and direct summands by Lemma~\ref{lem:perpendicular}(1), it is reflective by (a dual version of) \cite[Theorem~3.1]{CICS}. We will give a short argument for the convenience of the reader, as it simplifies to some extent when compared to the more general result of~\cite{CICS}.
So, let $M\in\ModA$ and consider $\Lcal$-preenvelopes $f\dd M\la L^0$ and $\Coker f\la L^1$. The sequence $M \overset{f}\la L^0 \overset{g}\la L^1$
induces by construction a right exact sequence of abelian groups
\begin{equation}\label{eq:construction of adjoint}
\Hom{R}{L^1}{L}\overset{g^*}\la\Hom{R}{L^0}{L}\overset{f^*}\la\Hom{R}{M}{L} \la 0
\end{equation}
for each $L\in\Lcal$. Let $K=\Ker g$ and denote by $k\dd K\hookrightarrow L^0$ the inclusion; then $f$ uniquely factors through $k$ by construction, and we claim that this factorization $\psi_M\dd M\la K$ is the required $\Lcal$-reflection.

First of all, $K\in\Lcal$ by Lemma~\ref{lem:perpendicular}(1). It remains to prove that $\Hom{A}{\psi_M}{L}$ is bijective for each $L\in\Lcal$. It is surjective since $\Hom{A}{\psi_M}{L}$ factors through $\Hom{A}{f}{L}$ and $f\dd M\la L$ is an $\Lcal$-preenvelope. To prove injectivity, notice first that $\psi_M\dd M\la K$ factors as $\psi_M=hf$ for some morphism $h\dd L\la K$, again because $f\dd M\la L$ is an $\Lcal$-preenvelope. Since then $f=k\psi_M=khf$, or in other words $(\id_{L^0}-kh)f=0$, there is a morphism $s\dd L^1\la L^0$ such $\id_{L^0}-kh=sg$ by the exactness of~\eqref{eq:construction of adjoint}. It follows that $k-khk=sgk=0$ (since $k$ is the kernel of $g$) and consequently $hk=\id_K$ (since $k$ is a monomorphism).
If now $t\dd K\la L$ is a morphism with $L\in\Lcal$ such that $t\psi_M=0$, then also $thf=0$, so $th=ug$ for some $u\dd L^1\la L$, using the exactness of~\eqref{eq:construction of adjoint} again. However, then $t=thk=ugk=0$, as required.
\end{proof}

\begin{remark}
The existence of the left adjoint in Proposition~\ref{prop:perp (co)relflection}(1) also follows by a more abstract argument as in~\cite[Examples 1.3(4)]{PosPerp}, using~\cite[Theorem and Corollary 2.48]{AR} and the fact that $\Scal^{\perp_{0,1}}$ is also closed under $\lambda$-filtered colimits for a large enough regular cardinal $\lambda$. However, that method does not dualize, so we cannot use it to prove Proposition~\ref{prop:perp (co)relflection}(2).
\end{remark}

\section{Silting ring epimorphisms} \label{sec:silting and epi}

As observed in~\cite{AMV2} (building on~\cite{GL,AS1,AA}),
there is a close connection between (partial) silting modules on the one hand and epimorphisms in the category of rings on the other. Given a partial silting $A$-module, there is always an associated ring epimorphism $A\longrightarrow B$, which can be viewed as a certain generalized localization of $A$. Conversely, for certain ring epimorphisms we can go back. The aim of this and the following sections is to understand in detail conditions under which such a reconstruction is possible.

\subsection{Inverting maps between projective modules}
We will be particularly interested in universal ring epimorphisms that make a chosen maps between projective modules invertible. We will start with the following lemma.

\begin{lemma}\label{lem:perp to proj}
Let $\Sigma$ be a set of morphisms in $\ProjA$ and denote  
\begin{align*}
\Xcal_\Sigma &= \{X\in \ModA \mid \Hom{A}{\sigma}{X} \text{ is bijective for each } \sigma\in\Sigma \} \\
&= \{X\in \ModA \mid \Hom{\DerA}{C_\sigma}{X[1]}=0=\Hom{\DerA}{C_\sigma}{X} \textit{ for each } \sigma\in\Sigma \}.
\end{align*}
Then $\Xcal_\Sigma$ is a wide subcategory of $\ModA$ that is also closed under products. The inclusion functor $\Xcal\longrightarrow\ModA$ admits a left adjoint.
\end{lemma}

\begin{proof}
Consider the subcategory $\widetilde{\Ccal} = \{C_\sigma\mid\sigma\in\Sigma\} \cup \{C_{\id_A}\}$ in $\MorA$. Notice that for each $\gamma\dd M_1\la M_0$, there is an isomorphism $\Hom{\MorA}{C_{\id_A}}{C_\gamma}\cong M_1$, so $M_1=0$ for any $C_\gamma\in\widetilde{\Xcal}_\Sigma = \widetilde{\Ccal}^{\perp_{0,1}}$. Together with Proposition~\ref{prop:morphismcat}(2), this implies that the class $\widetilde{\Xcal}_\Sigma$ in $\MorA$ contains precisely the objects of the form $C_{0\to X}$ with $X$ in $\Xcal_\Sigma$ (see also \cite[Remark 6.10]{MS}). Since the projective dimension of all objects in $\widetilde{\Ccal}$ is at most one by the proof of \cite[Lemma 5.1]{MS}, the class $\widetilde{\Xcal}_\Sigma$ is wide and closed under products in $\MorA$ by Lemma~\ref{lem:perpendicular}(1) (using the fact that $\MorA\simeq\rmod{T_2(A)}$). It follows that $\Xcal_\Sigma$ has the same closure properties in $\ModA$.

Finally, the inclusion functor $\widetilde{\Xcal}_\Sigma\la\MorA$ has a left adjoint by Proposition~\ref{prop:perp (co)relflection}(1). Given any $M\in\ModA$ and an $\widetilde{\Xcal}_\Sigma$-reflection $C_{(0\to M)} \longrightarrow C_{(0\to X_M)}$ in $\MorA$, it is straightforward to check that the underlying morphism $M\longrightarrow X_M$ is an $\Xcal$-reflection in $\ModA$.
\end{proof}

A dual analogous lemma for tensor product alsho holds:

\begin{lemma}\label{lem:tor to proj}
Let $\Sigma$ be a set of morphisms in $\ProjA$ and denote  
\begin{align*}
\Ycal_\Sigma &= \{Y\in \AMod \mid \sigma\otimes_AY \text{ is bijective for each } \sigma\in\Sigma \} \\
&= \{Y\in \AMod \mid \sigma\otimes^\mathbb{L}_AY = 0 \textit{ for each } \sigma\in\Sigma \}.
\end{align*}
Then $\Ycal_\Sigma$ is a wide subcategory of $\AMod$ that is also closed under coproducts. The inclusion functor $\Ycal\longrightarrow\AMod$ admits a right adjoint.
\end{lemma}

\begin{proof}
Consider again the subcategory $\widetilde{\Ccal} = \{C_\sigma\mid\sigma\in\Sigma\} \cup \{C_{\id_A}\}$ in $\MorA$ and also
$\widetilde{\Ycal}_\sigma=\widetilde{\Ccal}^{\torp{0,1}}\subseteq\AMor$. Now $C_{\id_A}\otimes_{T_2(A)}C_\beta\cong N^1$ by~\eqref{eq:tensor maps} for any morphism $\beta\colon N^0\to N^1$ of left $A$-modules, so
\[ \widetilde{\Ccal}^{\torp{0,1}} = \{ C_{N\to0} \mid N\in\Ycal_\Sigma \} \]
by Proposition~\ref{prop:morphismcat}(3).
The rest of the proof is completely analogous to the proof of Lemma~\ref{lem:perp to proj}, using Lemma~\ref{lem:perpendicular}(2) and Proposition~\ref{prop:perp (co)relflection}(2) this time.
\end{proof}

\begin{remark} \label{rem:silting ring epi}
In the cases where the class $\Xcal_\Sigma$ is also closed under coproducts or $\Ycal_\Sigma$ is closed under products (e.~g.\ if $\Sigma$ consists of maps between finitely generated projective modules), then $\Xcal_\Sigma$ or $\Ycal_\Sigma$ is the essential image of the forgetful functor $\lambda_*\colon \ModB \longrightarrow \ModA$ or $\lambda_*\colon \BMod \longrightarrow \AMod$, respectively, for a certain ring epimorphism $\lambda\colon A\longrightarrow B$ (see Theorem \ref{thm:epicl}).

In general, the class $\Xcal_\Sigma$ is the category of modules over an accessible additive monad (see \cite[Examples 1.3(4)]{PosPerp} for details) and $\Ycal_\Sigma$ is a Grothendieck category.
\end{remark}

We will show now that if $\Xcal_\Sigma$ is closed under coproducts, then $\Ycal_\Sigma$ is automatically closed under products and that they correspond to the same ring epimorphism. For this, we need the concept of definable subcategory.

\begin{definition}\label{def:definable}
A full subcategory $\Dcal$ of $\ModA$ or $\AMod$ is called definable if it is closed under direct limits, pure submodules and products.
\end{definition}

There is an order-preserving bijective correspondence between the class definable subcategories of $\ModA$ and the class of definable subcategories of $\AMod$ given by the so-called elementary duality. For our purposes, the correspondence can be described in the easiest way as follows. If a subcategory $\Dcal$ of $\ModA$ is definable, we put 
\begin{equation}\label{eq:dual definable}
\Dcal^+ = \{ N\in\AMod \mid N^+\in\Dcal \},
\end{equation}
where $(-)^+=\Hom{\Z}{-}{\QQ/\Z}$ denotes the character dual. Then $\Dcal^+$ is the corresponding definable subcategory of the left modules and $(\Dcal^+)^+=\Dcal$. We refer to \cite[\S3.4.2]{P} for details.

\begin{lemma} \label{lem:elementary dual ring epi}
Let $\lambda\colon A \longrightarrow B$ be a ring epimorphism and denote by $\Xcal_B$ the essential image of $\lambda_*\colon \ModB \longrightarrow \ModA$ and by $\Ycal_B$ the essential image of $\lambda_*\colon \BMod \longrightarrow \AMod$. Then both $\Xcal_B$ and $\Ycal_B$ are definable in their respective module categories and are elementarily dual to each other, i.~e.\ $\Xcal_B^+ = \Ycal_B$ and $\Ycal_B^+ = \Xcal_B$.
\end{lemma}

\begin{proof}
The classes $\Xcal_B$ and $\Ycal_B$ are definable by \cite[Corollary 5.5.4]{P}, and related by the elementary duality by~\cite[Example 3.4.22]{P}.
\end{proof}

As an immediate consequence, we obtain the following.

\begin{proposition} \label{prop:left right generalised localization}
Let $\Sigma$ be a set of morphisms in $\ProjA$ and suppose that $\Xcal_\Sigma$ is closed under coproducts. Then the class
\begin{align*}
\Ycal_\Sigma &= \{Y\in \AMod \mid \sigma\otimes_AY \text{ is bijective for each } \sigma\in\Sigma \}
\end{align*}
is closed under products. Moreover, there exists a ring epimorphism $\lambda\colon A\longrightarrow B$ such that $\Xcal_\Sigma$ is the essential image of $\lambda_*\colon \ModB \longrightarrow \ModA$ and $\Ycal_\Sigma$ is the essential image of $\lambda_*\colon \BMod \longrightarrow \AMod$.
\end{proposition}

\begin{proof}
A ring epimorphism $\lambda$ such that $\Xcal_\Sigma$ is the essential image of $\lambda_*\colon \ModB \longrightarrow \ModA$ exists by Remark~\ref{rem:silting ring epi}. Since we have $(\sigma\otimes_AY)^+\cong \Hom{A}{\sigma}{Y^+}$ for each $Y$ in $\AMod$, we find that $Y$ lies in $\Ycal_\Sigma$ \iff $Y^+$ lies in $\Xcal_\Sigma$. Thus, $\Ycal_\Sigma$ is the essential image of $\lambda_*\colon \BMod \longrightarrow \AMod$ by Lemma~\ref{lem:elementary dual ring epi}.
\end{proof}

\begin{problem} \label{prob:generalised localization}
The previous lemma says that if $\Xcal_\Sigma$ is closed under coproducts, then $\Ycal_\Sigma$ is closed under products. It would be interesting to know if the converse is true as well. If $\Ycal_\Sigma$ is closed under products, we know that it is associated with a ring epimorphism $\lambda\colon A\la B$ by Remark~\ref{rem:silting ring epi}. A more refined version of this question is whether $\Xcal_\Sigma$ is necessarily the essential image of $\lambda_*\colon \ModB \longrightarrow \ModA$. Since $B\in\Ycal_\Sigma$, when viewed as a left $A$-module, we immediately see that $B^+\in\Xcal_\Sigma$. Since $\Xcal$ is closed under products and kernels and $B^+$ is an injective cogenerator of $\ModB$, it follows that $\im\lambda_*\subseteq\Xcal_B$. Thus, it remains to decide whether the converse inclusion holds.
\end{problem}

\subsection{Definition of silting epimorphisms}
Now we will characterize the situation where $\Xcal_\Sigma$ is induced by (a projective presentation of) a partial silting module.

\begin{proposition} \label{prop:loc at proj silting}
Let $\Sigma$ be a set of morphisms in $\ProjA$. 
Then $\Dcal_\Sigma$ is closed under coproducts if and only if $\Dcal_\Sigma$ is a silting class.

If this is the case, the class $\Xcal_\Sigma$ is wide and bireflective, and there exists a partial silting module $T_1$ with respect to a projective presentation $\sigma_1$ such that $\Dcal_\Sigma = \Dcal_{\sigma_1}$ and $\Xcal_\Sigma = \Xcal_{\sigma_1}$.
Moreover, the corresponding ring epimorphism $\lambda\colon A\longrightarrow B$ has the property that $\sigma\otimes_A B$ is an isomorphism for each $\sigma\in\Sigma$ and it is a universal ring homomorphism with this property.
\end{proposition}

\begin{proof}
The equivalence between $\Dcal_\Sigma$ being closed under coproducts and being a silting class follows from \cite[Theorem 6.3]{MS}.

The second statement regarding the existence of $\sigma_1$ follows essentially by the same argument as in \cite[Theorem 6.7]{MS}. The assumption there was that $\Sigma$ consists of maps between finitely generated projective modules, but, in fact, it is enough to assume that $\Dcal_\Sigma$ is closed under coproducts, as we do here.
More precisely, the same reduction as there implies that it suffices to show as in \cite[Theorem 3.4(1)]{MS} that there is a partial tilting object $\widetilde{T}_1$ in $\MorA$ such that $\widetilde{T}_1^{\perp_1} = \widetilde{\Ccal}^{\perp_1}$ and $\widetilde{T}_1^{\perp_{0,1}} = \widetilde{\Ccal}^{\perp_{0,1}}$, where $\widetilde{\Ccal}=\{C_\sigma\mid\sigma\in\Sigma\}\cup\{C_{\id_A}\}$ (the same class as in the proof of Lemma~\ref{lem:perp to proj}). Here, the difference from \cite[Theorem 3.4(1)]{MS} is that our assumptions do not ensure that the objects in $\widetilde{\Ccal}$ are finitely presented. However, the objects in $\widetilde{\Ccal}$ are of projective dimension at most one in $\MorA$ and the class
\[ \widetilde{\Ccal}^{\perp_1} = \{ C_\gamma \in \MorA \mid \Coker\gamma\in\Dcal_\Sigma \} \]
is assumed to be closed under coproducts. In particular, the class $\widetilde{\Ccal}^{\perp_1}$ is a tilting torsion class in $\MorA$ by \cite[Corollary 14.6]{GT}. This is all that is needed for the proof of \cite[Theorem 3.4(1)]{MS} to work. It remains to note that $\Xcal_\Sigma$ is bireflective, or equivalently by Lemma~\ref{lem:perp to proj}, closed under coproducts. However, this is clear from the fact that $\Xcal_\Sigma=\Dcal_\Sigma\cap T_1^{\perp_0}$ and both classes in the intersection are closed under coproducts (we assume that about $\Dcal_\Sigma$ and $T_1^{\perp_0}$ is a torsion-free class in $\ModA$).

To prove the last part, suppose that $\lambda\colon A \longrightarrow B$ is a ring epimorphism such that $\im\lambda_\ast = \Xcal_\Sigma$ and let $\mu\colon A \longrightarrow C$ be any ring homomorphism. We know from Proposition~\ref{prop:left right generalised localization} that $\sigma\otimes_AC$ is an isomorphism for each $\sigma$ in $\Sigma$ \iff $C$ is in the essential image of $\lambda_*\colon \BMod\longrightarrow\AMod$.
In such a case $C$ canonically becomes a $B$-$C$-bimodule and we have a ring homomorphism $B \longrightarrow \End{C_C}$ which sends $b\in B$ to $(b\cdot-)$. If we identify $\End{C_C}$ with $C$ via the isomorphism $c\longmapsto(c\cdot-)$, one directly checks that we have obtained a factorization of $\mu$ through $\lambda$ and this factorization is unique since $\lambda$ is a ring epimorphism.
\end{proof}

\begin{definition}[{\cite{AMV2}}] \label{def:silting epi}
A ring epimorphism arising from a partial silting module as in Proposition~\ref{prop:loc at proj silting} will be said to be a {\em right silting ring epimorphism}. An analogous situation where $\Sigma$ is a set of morphisms between projective left modules gives rise to {\em left silting ring epimorphisms}.
\end{definition}

\begin{remark} \label{rem:silting epi} 
(1) Thanks to Theorem~\ref{thm:epicl}, we have $\Tor{1}{A}{B}{B}=0$ for
every silting ring epimorphism $\lambda\dd A\longrightarrow B$. In general, however, $\lambda$ may not be a homological epimorphism in the sense of~\cite[Definition 4.5]{GL}, i.~e.~$\Tor{i}{A}{B}{B}$ can be non-trivial for $i>1$, see the examples in \cite[Section 4]{AMV2}.
However, if $A$ is a commutative noetherian ring, the condition that $\Tor{1}{A}{B}{B}=0$ implies even that $\lambda$ is a flat ring epimorphism, \cite[Proposition 4.5]{AMSTV}. The situation where $A$ is a general commutative ring will be discussed in detail in Section~\ref{sec:flat is silting} (see also item (3) in this remark).

\smallskip

(2) As shown by Schofield in \cite{Sch}, for every set $\Sigma$ of morphisms in $\projA$ there is a ring epimorphism $\lambda_\Sigma\colon A\longrightarrow A_\Sigma$, called the \emph{universal localization} of $A$ at $\Sigma$,  which is universal with respect to the property that $\sigma\otimes_A A_\Sigma$ is an isomorphism for every $\sigma$ in $\Sigma$.  This is in fact a special case of Proposition~\ref{prop:loc at proj silting} since in this case $\Dcal_\Sigma$ is clearly closed under coproducts, so Proposition~\ref{prop:loc at proj silting} generalizes~\cite[Theorem~6.7]{MS}.

\smallskip

(3) The condition that $\Dcal_\Sigma$ is closed under coproducts as in Proposition~\ref{prop:loc at proj silting} is strictly stronger than just asking that $\Xcal_\Sigma$ be closed under coproducts as in Proposition~\ref{prop:left right generalised localization}. Indeed, let $A$ be a commutative domain with a non-zero proper ideal $I$ such that
\begin{itemize}
\item $\Tor{i}{A}{A/I}{A/I}=0$ for all $i>0$ and
\item $I$ has projective dimension one over $A$.
\end{itemize}
This situation appeared in \cite[\S2]{Kel-smash}, see also~\cite[\S7]{KS}.
Then $\Tor{i}{A}{I}{A/I}=\Tor{i+1}{A}{A/I}{A/I}=0$ for $i=0,1$, so if $\sigma\colon P_1\longrightarrow P_0$ is a projective presentation of $I$, it follows that $\sigma\otimes_AA/I$ is an isomorphism.
Moreover, the class $\Xcal_\sigma=\{X\in\ModA\mid\Hom{A}{\sigma}{X}\}=I^{\perp_{0,1}}$ consists precisely of all modules annihilated by~$I$. Indeed, if $\Hom{A}{I}{X}=0$, then $\Hom{A}{A/I}{X}\to\Hom{A}{A}{X}\cong X$ is an isomorphism, so $XI=0$. Conversely, if $XI=0$, then $X$ is an $A/I$-module, $\Ext{>0}{A}{A/I}{X}\cong\Ext{>0}{A/I}{A/I}{X}=0$ by~\cite[Theorem 4.4]{GL} and so $X\cong\Rhom{A}{A/I}{X}$. Then
\[ \Rhom{A}{I}{X}\cong\Rhom{A}{I}{\Rhom{A}{A/I}{X}}\cong \Rhom{A}{I\lten{A}A/I}{X} = 0. \]
Thus, the class $\Xcal_\sigma$ is closed under coproducts and is associated with the ring epimorphism $\lambda\colon A \longrightarrow A/I$ in the sense of Proposition~\ref{prop:left right generalised localization}.
On the other hand, it will follow from Theorem~\ref{thm:flat is silting} that $\lambda$ cannot be a silting epimorphism since $A/I$ is not flat over $A$ (as $\Tor{1}{A}{A/aA}{A/I}\cong A/I\ne 0$ for each $0\ne a\in I$). 
\end{remark}

In general, it is not clear how to recognize whether a given ring epimorphism is silting, nor there seem to be an easy way to recover a partial silting module even when we know that the ring epimorphism in question is silting.
Note that every ring epimorphism $\lambda$ gives rise to an exact sequence 
$$\xymatrix{A\ar[r]^\lambda& B\ar[r]& \Coker{\lambda}\ar[r]& 0}.$$ 
One might ask whether $\Coker\lambda$ itself can be a partial silting module when $\lambda$ is a silting ring epimorphism. We provide three examples---one class of ring epimorphisms where this simple idea works, one example where it fails, and finally an ad hoc alternative construction which works in the second case. In the next section, we will provide a more systematic method for recovering a partial silting module from a ring epimorphism.

\begin{example}\label{expl:silting pd1} 
Let $\lambda\colon A\longrightarrow B$ be a ring epimorphism. 
If $\Tor{1}{A}{B}{B}=0$ and $B$ has projective dimension at most one over $A$, then $\lambda$ is a silting ring epimorphism. In fact, $\lambda$ as a map of right modules is quasi-isomorphic to a map $\sigma\dd P_1\la P_0$ between projective $A$-modules which provides a suitable projective presentation demonstrating the fact that $\Coker\lambda$ is partial silting, and the corresponding silting torsion class in $\ModA$ is $\Dcal_\sigma=\Gen B$.
See~\cite[Example 6.5]{MS}.

This applies, for example, to the situation where $A$ is commutative noetherian of Krull dimension at most one and $B$ is flat over $A$ (i.~e.\ $\lambda$ is a flat ring epimorphism). Indeed, then every flat $A$-module has projective dimension at most one by \cite[Corollaire~3.2.7]{RG}.
However, in higher Krull dimensions, the situation can become more complicated as illustrated in the next example.
\end{example}

\begin{example}\label{expl:example not partial silting}
If $A$ is a regular commutative noetherian ring and $\lambda\colon A\longrightarrow B$ is a ring epimorphism with $\Tor{1}{A}{B}{B}=0$, then $\lambda$ is a universal localization (in the sense of Remark~\ref{rem:silting epi}(2)) at a set of maps between finitely generated projective $A$-modules. This follows from~\cite[Proposition~4.5 and Theorem~5.13]{AMSTV}.
In particular, as seen in Proposition \ref{prop:loc at proj silting}, such a ring epimorphism is silting.

An easy example of a ring epimorphism satisfying the above conditions is the inclusion $\lambda$ of the regular local ring $A=\mathbb{K}[[x,y]]$ of Krull dimension 2 to its field of fractions $B=\mathbb{K}(\!(x,y)\!)$. We will show that the cokernel of $\lambda$ in $\ModA$, let us call it $T$, is not partial silting with respect to any projective presentation. Note first that $T$ has
\begin{itemize}
\item injective dimension at most $1$ (since $B$ is the injective envelope of $A$ in $\ModA$ and $A$ is a ring of global dimension 2) and
\item projective dimension $2$ (see \cite[Corollary 2.59]{Os}).
\end{itemize}
Let $\sigma$ be an arbitrary projective presentation of $T$ and complete it to a projective resolution in $\ModA$ as follows:
$$\xymatrix{0 \ar[r]& P_2 \ar[r]^\mu& P_1 \ar[r]^\sigma& P_0 \ar[r]& T \ar[r]& 0}. $$
Clearly, since $T$ has injective dimension $1$, $\Ext{2}{A}{T}{T}=0$ and therefore $\Hom{A}{\mu}{T}$ is surjective. Moreover, since $A$ is local, $P_2$ is a free module, and therefore we have $\Hom{A}{P_2}{T}\neq 0$. This shows that $\Hom{A}{\sigma}{T}$ cannot be surjective and thus $T$ cannot belong to $\Dcal_\sigma$, whatever the choice of projective presentation $\sigma$.
\end{example}

\begin{example}\label{expl:partial silting univ loc}
Given a commutative ring $A$ and a set $\Sigma$ of maps in $\projA$, we will show an explicit construction of a map $\sigma\dd P_1\la P_0$ in $\ProjA$ that is a resolution of a partial silting module $T_\Sigma$ with $\Dcal_\Sigma=\Dcal_{\sigma}$ and $\Xcal_\Sigma=\Xcal_{\sigma}$. This illustrates how Proposition~\ref{prop:loc at proj silting} works for universal localizations from Remark~\ref{rem:silting epi}(2) and remedies Example~\ref{expl:example not partial silting}. The method is formally similar to~\cite[Construction 4.5]{AH}, which was in turn inspired by the constructions of Fuchs~\cite{FuchsDiv} and Salce~\cite{SalceTilt}.

Let us choose an indexing set $I$ for $\Sigma$, so the maps in $\Sigma$ will be denoted by $\sigma_i\dd P_{1,i}\la P_{0,i}$, and we will denote by $I^+=\coprod_{n\ge 1} I^n$ the set of non-empty finite sequences of elements of $I$. The length of $\underline{i}\in I^+$, denoted by $\ell(\underline{i})$, is the unique $n\ge 1$ such that $\underline{i}\in I^n$, and given a sequence $\underline{i}=(i_0,\dots,i_{n-1})$ of length greater than $1$, we denote $\underline{i}^-=(i_0,\dots,i_{n-2})$. For all sequences $\underline{i}=(i_0,\dots,i_{n-1})\in I^+$, we inductively define morphisms $\sigma_{\underline{i}}\dd P_{1,\underline{i}}\la P_{0,\underline{i}}$ in $\projA$ as follows:
\begin{itemize}
\item $\sigma_{\underline{i}}=\sigma_{i_0}$ for $n=1$, and
\item $\sigma_{\underline{i}}=\sigma_{i_{n-1}}\otimes_A\Hom{A}{P_{1,i_{n-1}}}{P_{0,\underline{i}^-}}$ for $n>1$.
\end{itemize}
Note that the $A$-modules $\Hom{A}{P_{1,i_{n-1}}}{P_{0,\underline{i}^-}}$ are finitely generated projective, so each $\sigma_{\underline{i}}$ is a summand of a finite direct sum of copies of $\sigma_{i_{n-1}}\in\Sigma$.

Now we can construct the map $\sigma\colon P_1\la P_0$ as follows. The domain and the codomain will be taken as, respectively, $P_1=\bigoplus_{\underline{i}\in I^+}P_{1,\underline{i}}$ and $P_0=\bigoplus_{\underline{i}\in I^+}P_{0,\underline{i}}$. Any map $P_1\la P_0$ is determined by its components $P_{1,\underline{i}}\la P_{0,\underline{j}}$, and the only non-zero components of $\sigma$ will be
\begin{itemize}
\item $\sigma_{\underline{i}}\dd P_{1,\underline{i}}\la P_{0,\underline{i}}$ for each $\underline{i}\in I^+$, and
\item whenever the length of $\underline{i}$ is greater than $1$, the negative of the canonical evaluation maps $\varepsilon_{\underline{i}}\colon P_{1,\underline{i}}=P_{1,i_{n-1}}\otimes_A\Hom{A}{P_{1,i_{n-1}}}{P_{0,\underline{i}^-}}\la P_{0,\underline{i}^-}$ that sends $p\otimes f$ to $-f(p)$.
\end{itemize}
So we put $T_\Sigma=\Coker\sigma$ and we also denote by $\pi\colon P_0\la T_\Sigma$ the canonical projection. By construction, we have the equalities
\begin{equation}\label{eq:divisibility-by-sigma}
\pi(\sigma_{i_{n-1}}(p)\otimes f) = \pi(f(p))
\end{equation}
in $T_\Sigma$ for each $\underline{i}\in I^+$ of length greater than $1$, $p\in P_{1,i_{n-1}}$ and a morphism $f\colon P_{1,i_{n-1}}\la P_{0,\underline{i}^-}$.

It remains to prove that $\Dcal_\Sigma=\Dcal_{\sigma}$ and $\Xcal_\Sigma=\Xcal_{\sigma}$, and for this we will use the argument from~\cite[Theorems~3.4(1) and~6.7]{MS}.
Given any $n\ge 0$, we put $P_1^{\le n}=\bigoplus_{\underline{i}\in I^+,\,\ell(i)\le n}P_{1,\underline{i}}$ and $P_0^{\le n}=\bigoplus_{\underline{i}\in I^+,\,\ell(i)\le n}P_{0,\underline{i}}$. These are summands of $P_1$ and $P_0$, respectively, and the map $\sigma$ introduced above restricts to maps $\sigma^{\le n}\colon P_1^{\le n}\la P_0^{\le n}$. Moreover, in the notation of~\S\ref{subsec:torsion and silting}, we have by construction a sequence of subobjects of $C_\sigma\in\MorA$,
\begin{equation}\label{eq:filtration-univ-loc}
0=C_{\sigma^{\le 0}} \subseteq C_{\sigma^{\le 1}} \subseteq C_{\sigma^{\le 2}} \subseteq \cdots \subseteq C_\sigma,
\end{equation}
whose union is all of $\sigma$. Moreover, by construction, we have for each $n\ge 1$ that
\[
C_{\sigma^{\le n}}/C_{\sigma^{\le n-1}} \cong \bigoplus_{\underline{i}\in I^n}C_{\sigma_{\underline{i}}} \in \Add\{C_{\sigma_i} \mid \sigma_i\in\Sigma \}.
\]

Next, we claim that $T:=\Coker\sigma\in\Dcal_\Sigma$ in $\ModA$. To see that, let $\sigma_i\dd P_{1,i}\la P_{0,i}$ be an element of $\Sigma$ and $g\colon P_{1,i}\la T_\Sigma$ be a morphism. Then $g$ lifts to a morphism $f\colon P_{1,i}\la P_0 = \bigoplus_{\underline{j}\in I^+}P_{0,\underline{j}}$, and such $f$ can only have finitely many non-zero components $f_{\underline{j}}\colon P_{1,i}\la P_{0,\underline{j}}$. Given any such component $f_{\underline{j}}$ with $\underline{j}=(i_0,\dots,i_{n-1})$, we denote $\underline{j}^+=(i_0,\dots,i_{n-1},i)$ and consider a morphism
\[
h_{\underline{j}^+}\colon P_{0,i}\la P_{0,\underline{j}^+}=P_{0,i}\otimes_A\Hom{A}{P_{1,i}}{P_{0,\underline{j}}},
\quad
q \mapsto q\otimes f_{\underline{j}}.
\]
Now, given any $p\in P_{1,i}$, we have $\pi(f_{\underline{j}}(p))=\pi(\sigma_i(p)\otimes f_{\underline{j}})=\pi(h_{\underline{j}^+}(\sigma_i(p)))$ by~\eqref{eq:divisibility-by-sigma}, so the following triangle commutes.
\[
\xymatrix{
P_{1,i} \ar[r]^{\sigma_i} \ar[d]_{\pi f_{\underline{j}}=g_{\underline{j}}} & P_{0,i} \ar[dl]^{\pi h_{\underline{j}^+}} \\
T_\Sigma
}
\]
If we sum all the above maps $\pi h_{\underline{j}^+}\colon P_{0,i}\la T_\Sigma$ such that $f_{\underline{j}}\ne 0$, we obtain a factorization of $g$ through $\sigma_i$. This proves the claim.

To conclude, we invoke~\cite[Lemma~5.4]{MS} which tells us that $C_\sigma\in\{C_{\sigma_i}\mid \sigma_i\in\Sigma\}^{\perp_1}$, and along with the filtration from~\eqref{eq:filtration-univ-loc},  the above construction is just an explicit way to obtain the direct sum of the approximation sequences at the beginning of the proof of \cite[Theorem~3.4(1)]{MS} applied to the set of finitely presented objects $\{C_{\sigma_i}\mid \sigma_i\in\Sigma\}$ in $\MorA\simeq\rmod{T_2(A)}$. Now we can just follow the argument of ~\cite[Theorem~6.7]{MS}, which invokes \cite[Theorem~3.4(1)]{MS} for this set of objects.
\end{example}

\section{Ring epimorphisms: from flat to silting} \label{sec:flat in general}

In this section, we establish a criterion for a flat ring epimorphism to be a silting ring epimorphism. As a consequence, all flat ring epimorphisms originating in a commutative ring turn out to be silting epimophisms (Corollary~\ref{cor:commutative flat is silting}).
Since flat epimorphisms are closely related to Gabriel filters by Theorem~\ref{thm:perfectGT and ringepi}, we begin with an analysis of finiteness conditions on them.

\subsection{Finiteness conditions on Gabriel filters}
There are three finiteness conditions that we will consider on Gabriel filters. The strongest of them has already been discussed in Definition~\ref{def:perfect}: a Gabriel filter is perfect if and only if its associated Giraud subcategory is closed under arbitrary colimits. We will also consider two weaker finiteness notions:

\begin{definition}\label{def:reminder}\label{def:Giraud finite}
Let $\Gcal$ be a Gabriel filter of left ideals of $A$. We say that 
\begin{enumerate}
\item $\mathcal{G}$ is of \emph{finite type} if it has a filter basis consisting of finitely generated ideals or, equivalently (\cite[Lemma 2.4]{H}), if its associated hereditary torsion pair $(\Tcal_\Gcal,\Fcal_\Gcal)$ is of finite type (i.~e. $\Fcal_\Gcal$ is closed under direct limits);
\item $\mathcal{G}$ is \emph{Giraud-finite} if its associated Giraud subcategory $\Xcal_\Gcal$ is closed under direct limits.
\end{enumerate}
 \end{definition}
 
It is clear that every perfect Gabriel filter is Giraud-finite, and it is well-known that every perfect Gabriel filter if of finite type (\cite[Proposition 3.4]{St}).  The following statement says, in particular, that every Giraud-finite Gabriel filter is of finite type.

\begin{proposition}\label{prop:Giraud finite}\label{rem:Giraud finite}
The following are equivalent for a left Gabriel filter $\Gcal$:
\begin{enumerate}
\item The Giraud subcategory $\Xcal_\Gcal$ is closed under direct limits in $\AMod$.
\item $\Gcal$ is of finite type and each finitely generated left ideal $I\in\Gcal$ has a syzygy $\Omega(I)$ which admits a short exact sequence
\[ 0 \longrightarrow F_I \longrightarrow \Omega(I) \longrightarrow T_I \longrightarrow 0 \]
with $F_I$ finitely generated and $T_I\in\Tcal_G$.
\end{enumerate}
\end{proposition}

\begin{proof}
Recall the functor $i\cong\Hom{\Xcal_\Gcal}{A_\Gcal}{-}\colon\mathcal X_\Gcal\hookrightarrow\lmod{A}$ and note that the characterization of $M_\Gcal$ through the exact sequence~\eqref{eq:reflect} shows that (1) holds if and only if $i$ commutes with direct limits if and only if $A_\Gcal$ is a finitely presentable object of $\Xcal_\Gcal$. Moreover, since the $\Gcal$-torsion part $t(M)$ of $M$ can be computed as $\Ker\psi_M$ in~\eqref{eq:reflect}, (1) implies that $t\colon \AMod\longrightarrow\AMod$ commutes with direct limits. The latter is equivalent to $\Fcal_\Gcal$ being closed under direct limits, or equivalently to the finite type of $\Gcal$.

Suppose now that (1) holds, let $I\in\Gcal$ be a finitely generated left ideal, and consider an exact sequence
\[
0 \longrightarrow \Omega(I) \longrightarrow A^n \longrightarrow A \longrightarrow A/I \longrightarrow 0
\]
in $\AMod$. If we apply the exact functor $\ell\colon \AMod \longrightarrow \Xcal_\Gcal$ left adjoint to $i$, we obtain a sequence
\[ 0 \longrightarrow \Omega(I)_\Gcal \longrightarrow A_\Gcal^n \longrightarrow A_\Gcal \longrightarrow 0 \]
 which is exact in the abelian category $\Xcal_\Gcal$ (although not necessarily in $\AMod$). Since $A_\Gcal$ is finitely presentable in $\Xcal_\Gcal$, \cite[Proposition V.3.4]{St} implies that $\Omega(I)_\Gcal$ is a finitely generated object of $\Xcal_\Gcal$. In particular, if we express $\Omega(I) = \varinjlim_j F_j$ as the direct union of its finitely generated submodules, then $\Omega(I)_\Gcal$ coincides with the direct union $\varinjlim_j (F_j)_\Gcal$ in the category $\Xcal_\Gcal$ and therefore $(F_{j_0})_\Gcal \cong \Omega(I)_\Gcal$ for some index $j_0$. As this is essentially the image of the inclusion $F_{j_0}\subseteq \Omega(I)$ under the functor $\ell$, we have $\ell(\Omega(I)/F_{j_0}) = 0$ or, in other words, $\Omega(I)/F_{j_0}$ lies in $\Tcal_\Gcal$. Hence, we can put $F_I = F_i$ to obtain an exact sequence for $\Omega(I)$ as in (2).

Suppose conversely that (2) holds. In this case, we can use a straightforward variation of the proof of \cite[Corollary 3.10]{HSt}. By assumption, for each finitely generated $I\in\Gcal$ there is a complex
$P_I=((P_I)_2 \longrightarrow (P_I)_1 \longrightarrow A)$
of finitely generated projective modules, with $A$ in homological degree zero, such that $H_0(P_I)=A/I$ and with $H_1(P_I)$ in $\Tcal_\Gcal$.
For any $M$ in $\AMod$, it is clear that $M$ lies in $\Fcal_\Gcal$ if and only if $H^0\Hom{A}{P_I}{M}=0$. 
In addition, note that if $M$ is $\Gcal$-torsion-free, then the morphism $\Hom{A}{\Omega(I)}{M}\longrightarrow\Hom{A}{F_I}{M}$ is injective and hence $\Ext{1}{A}{A/I}{M}\cong H^1\Hom{A}{P_I}{M}$. It then follows that
\[ \Xcal_\Gcal=\{M\in \AMod\mid H^0\Hom{A}{P_I}{M}=0=H^1\Hom{A}{P_I}{M} \text{ for\ all } I\in\Gcal\}, \]
which immediately implies (1).
\end{proof}

\begin{corollary} \label{cor:Giraud finite}
For any ring $A$, we have inclusions
\[
\left\{
\begin{matrix}
\textrm{Perfect Gabriel}\\
\textrm{filters on $A$}
\end{matrix}
\right\}
\subseteq
\left\{
\begin{matrix}
\textrm{Giraud-finite Gabriel}\\
\textrm{filters on $A$}
\end{matrix}
\right\}
\subseteq
\left\{
\begin{matrix}
\textrm{Finite-type Gabriel}\\
\textrm{filters on A}
\end{matrix}
\right\}
\]
and, moreover, the last inclusion is an equality whenever $A$ is left coherent or $A$ is commutative.
\end{corollary}

\begin{proof}
If $A$ is left coherent, then each finitely generated left ideal $I$ is finitely presented, and we can choose $F_I=\Omega(I)$, proving that any finite type Gabriel filter is Giraud-finite. If $A$ is commutative, we can consider the Koszul complex
\[
K_\bullet(\mathbf{x}) =
\bigotimes_{i=1}^n(A \overset{x_i}\la A) =
\cdots \la A^{\genfrac(){0pt}{3}{n}{2}} \la A^n \stackrel{\mathbf{x}}\la A \la 0.
\]
for any finite sequence $\mathbf{x}=(x_1,x_2,\dots,x_n)$ of generators of $I$. It is well-known that the homology of $K_\bullet(\mathbf{x})$ is annihilated by $I$ (see~\cite[Theorem 16.4]{M}), so we can take $\Omega(I)=\Ker(A^n \overset{\mathbf{x}}\longrightarrow A) = Z_1(K_\bullet(\mathbf{x}))$ and $F_I=B_1(K_\bullet(\mathbf{x}))$. This allows us to show that any finite type Gabriel filter is Giraud-finite. 
\end{proof}

\subsection{Gabriel filters and silting classes}
For a Gabriel filter $\Gcal$, we will, in addition to the hereditary torsion pair $(\Tcal_\Gcal,\Fcal_\Gcal)$ and the Giraud subcategory $\Xcal_\Gcal$ in $\AMod$, also consider the class of $\mathcal{G}$-\emph{divisible} modules 
\begin{align*}
\Dcal_\Gcal&=\{M \in \ModA \mid  MI=M\text{ for all } I\in\Gcal\} \\
&=\{M \in \ModA \mid  M\otimes_A{A/I}=0\text{ for all } I\in\Gcal\}.
\end{align*}

\begin{remark}\label{rem:divisible}
A left Gabriel filter $\mathcal{G}$ is perfect if and only if $A_\Gcal$ is a $\Gcal$-divisible (right) $A$-module, cf.~\cite[Proposition XI.3.4]{St}. 
\end{remark}

It turns out that these divisibility classes are very closely related to silting classes.

\begin{proposition}\label{divisible silting}
If $A$ is a ring and $\Gcal$ is a left Gabriel filter of finite type, then the class $\Dcal_\Gcal$ is a silting class. Moreover, if $A$ is commutative, every silting class in $\ModA$ has this form.
\end{proposition}

\begin{proof}
The first part follows from~\cite[Corollaries~3.6 and~4.1]{AH}, and the case for commutative rings from~\cite[Theorem~4.7]{AH}.
\end{proof}

We are now ready to tackle the problem of which flat ring epimorphisms come from silting modules. We first start with a special case and then reduce the general situation to it.

\begin{proposition} \label{prop:cnt based is silting}
Let $\Gcal$ be a perfect Gabriel filter of left ideals of a ring $A$ with a countable filter basis. Then $\lambda\colon A \longrightarrow A_\Gcal$ is a right silting ring epimorphism with the corresponding silting class $\Gen A_\Gcal = \Dcal_\Gcal$.
\end{proposition}

\begin{proof}
The projective dimension of $A_\Gcal$ as a right $A$-module is at most one by \cite[Theorem~8.5]{Pos} and, hence, $\lambda$ is a silting ring epimorphisms and $\Coker\lambda$ is a corresponding partial silting module whose silting class is $\Gen A_\Gcal$ by Example~\ref{expl:silting pd1}.
It remains to prove that $\Gen A_\Gcal=\Dcal_\Gcal$ or, equivalently, by~\cite[Proposition~3.4]{AH}, that $\Cogen A_\Gcal^+=\Dcal_\Gcal^+$ (recall~\eqref{eq:dual definable}). However, $X\in\Dcal_\Gcal$ if and only if $0=(X\otimes_AA/I)^+\cong\Hom{A}{A/I}{X^+}$ for all $I\in\Gcal$ if and only if $X^+\in\Fcal_\Gcal$, so $\Dcal_\Gcal^+=\Fcal_\Gcal$ and we should show that $\Cogen A_\Gcal^+=\Fcal_\Gcal$ in $\AMod$. This is standard---on the one hand, $A_\Gcal^+\in\Xcal_\Gcal\simeq\lmod{A_\Gcal}$, so $\Cogen A_\Gcal^+\subseteq\Fcal_\Gcal$. On the other hand, $X\to A_\Gcal\otimes_A X$ is injective for each $X\in\Fcal_\Gcal$ by~\cite[Proposition~XI.3.4(f)]{St} and the left $A_\Gcal$-module $A_\Gcal\otimes_AX$ certainly belongs to $\Cogen A_\Gcal^+$, so also $X\in\Cogen A_\Gcal^+$.
\end{proof}

\subsection{From flat to silting}

To obtain a more general result on when a flat ring epimorphism is silting, we first record the following lemma, which generalizes~\cite[Lemma 3.8]{Pos} and is closely related to \cite[Lemma 3.6 and Remark 3.7]{Pos}.

\begin{lemma}\label{lem:direct limit of reflections}
Let $\Gcal = \bigcup_{j \in J}\Gcal_j$ be a direct union of Giraud-finite Gabriel filters $\Gcal_j$. Then $\Gcal$ itself is a Giraud-finite Gabriel filter, and for every $M$ in $\AMod$, we have a canonical isomorphism $M_\Gcal\cong \varinjlim_{j\in J} M_{\Gcal_j}$.
\end{lemma}

\begin{proof}
That $\Gcal$ is a Gabriel filter follows from~\cite[Lemma 3.1]{Pos} and it is Giraud-finite since $\Xcal_\Gcal = \bigcap_{j\in J}\Xcal_{\Gcal_j}$.
Given any $M$, we have a direct system of exact sequences
\[ 0\longrightarrow t_j(M)\longrightarrow M\stackrel{(\psi_j)_M}{\longrightarrow} M_{\Gcal_j}\longrightarrow \Coker(\psi_j)_M\longrightarrow 0 \]
indexed by $j\in J$ such that $t_j(M)$ and $\Coker(\psi_j)_M$ lie in $\Tcal_{\Gcal_j} (\subseteq \Tcal_\Gcal)$ and $M_{\Gcal_j}$ lies in $\Xcal_{\Gcal_j}$. Since $M_{\Gcal_k}$ belongs to $\Xcal_{\Gcal_k}\subseteq\Xcal_{\Gcal_j}$ for each $k\ge j$ in $J$ and all $\Xcal_{\Gcal_j}$ are closed under direct limits, we find that $\varinjlim M_{\Gcal_j}$ lies in $\bigcap_{j\in J}\Xcal_{\Gcal_j} = \Xcal_\Gcal$. Thus, as the direct limit map $\varinjlim (\psi_j)_M\colon M \la\varinjlim M_{\Gcal_j}$ has the kernel and cokernel in $\Tcal_\Gcal$ and the codomain in $\Xcal_\Gcal$, it is an $\Xcal_\Gcal$-reflection and as such it is canonically isomorphic to $\psi_M\colon M \longrightarrow M_\Gcal$
by the discussion below~\eqref{eq:reflect}.
\end{proof}

Now we can state the main result of the section, which relates flat epimorphisms to silting ones, and its immediate corollary for commutative rings.

\begin{theorem}\label{thm:from flat to silting}
Let $\lambda\colon A\longrightarrow B$ be a right flat ring epimorphism and $\Gcal = \{ I\subseteq A\mid B\cdot\lambda(I) = B \}$ be the corresponding left Gabriel filter on $A$. Suppose that the following conditions hold:
\begin{enumerate}
\item For each $I\in\Gcal$ there exists a subfilter $\Gcal_I\subseteq\Gcal$ of left ideals with a countable filter basis such that $(I:x)=\{a\in A\mid ax\in I\}\in\Gcal_I$ for each $x\in A$.
\item Each countably based Gabriel subfilter of finite type $\Gcal'\subseteq\Gcal$ is Giraud-finite.
\end{enumerate}
Then $\lambda$ is a right silting ring epimorphism.
\end{theorem}
\begin{proof}
Using results from~\cite[\S3]{Pos}, we will reduce the problem to one about countably based perfect Gabriel filters. First of all, \cite[Corollary 3.5]{Pos} tells us that assuming (1) we can express $\Gcal$ as a direct union of $\bigcup_{j\in J} \Gcal_j$ of (not necessarily perfect) Gabriel filters $\Gcal_j$ which are of finite type and have countable filter bases. Moreover, this union is countably directed, that is, any countable family of $\Gcal_j$ is contained in $\Gcal_{j'}$ for some $j'$ in $J$.

Secondly, we claim that there is a cofinal subsystem of countably based perfect Gabriel subfilters of $\Gcal$.
This can be shown using assumption (2) and a slight variation of the proof of \cite[Proposition 3.9]{Pos}. To that end, fix $j_0\in J$ and a countable basis $(I_n\mid n<\omega)$ of finitely generated left ideals of $\Gcal_{j_0}$. Given any $n<\omega$, we have $BI_n = B$ by Remark~\ref{rem:divisible} and this is equivalent to the existence of $k_n\ge 0$, $b_1,\dots,b_{k_n}\in B$ and $x_1,\dots,x_{k_n}\in I_n$ such that $1_B=\sum_{j=1}^{k_n}b_jx_j$.
Since $B=\varinjlim_{j\in J}A_{\Gcal_j}$ by Lemma~\ref{lem:direct limit of reflections}, there exist $j_{0,n}\in J$ and elements $a_1,\dots,a_{k_n}\in A_{\Gcal_{j_{0,n}}}$ such that $1_{A_{\Gcal_{j_{0,n}}}}=\sum_{j=1}^{k_n}a_jx_j$, or in other words $A_{\Gcal_{j_{0,n}}} I_n = A_{\Gcal_{j_{0,n}}}$. Now we can choose $j_1\ge j_0$ so that $\Gcal_{j_1}$ contains $\Gcal_{j_{0,n}}$ for each $n<\omega$. To summarize, we have found for any $J_0\in J$ a countably based Gabriel filter $\Gcal_{j_1}\supseteq\Gcal_{j_0}$ that has the property that $A_{\Gcal_{j_1}}$ is $\Gcal_{j_0}$-divisible.
If we repeat the same procedure countably many times, we obtain a chain $j_0 \le j_1 \le j_2 \le \cdots$ such that $A_{\Gcal_{j_{m+1}}}$ is $A_{\Gcal_{j_m}}$-divisible for each $m<\omega$. Let $\Gcal' = \bigcup_{m<\omega}\Gcal_{j_m}$. This is a countably based Gabriel filter and, by Lemma~\ref{lem:direct limit of reflections}, assumption (2) and the construction, $A_{\Gcal'}$ is $\Gcal'$-divisible, hence perfect by Remark~\ref{rem:divisible}. This proves the claim.

Finally, the flat epimorphism $\lambda'\colon A \longrightarrow A_{\Gcal'}$ is silting by Proposition~\ref{prop:cnt based is silting} for each countably based perfect Gabriel subfilter $\Gcal'\subset\Gcal$
(and the corresponding silting map $\sigma'\colon P_1 \longrightarrow P_0$ with $\Xcal_{\sigma'} = \Xcal_{\Gcal'}$ can be taken as a suitable projective presentation of $\Coker\lambda'$). Since $\Xcal_\Gcal = \bigcap_{\Gcal'}\Xcal_{\Gcal'}$, where the intersection runs through all countably based perfect Gabriel subfilters $\Gcal'\subseteq\Gcal$, also $\lambda\colon A \longrightarrow B$ is a silting ring epimorphism by Proposition~\ref{prop:loc at proj silting} (with the corresponding silting class $\Dcal_\Gcal=\bigcap_{\Gcal'}\Dcal_{\Gcal'}$).
\end{proof}

\begin{corollary}\label{cor:commutative flat is silting}
Any flat ring epimorphism originating in a commutative ring is a silting epimorphism.
\end{corollary}

\begin{proof}
We check that a Gabriel filter $\Gcal$ corresponding to a flat epimorphism from a commutative ring $A$ satisfies the assumptions of the theorem. For each $I$ in $\Gcal$, we can take for $\Gcal_I$ in (1) the principal filter induced by $I$ (i.~e. the set of ideals of $A$ containing $I$). Since $A$ is commutative, it is then clear that $(I:x)$ contains $I$ and, therefore, lies in $\Gcal_I$ for any $x\in A$. Condition (2) was discussed in Corollary~\ref{cor:Giraud finite}, and the result follows.
\end{proof}

An explicit description of an actual partial silting module whose existence is ensured by Theorem~\ref{thm:from flat to silting} or Corollary~\ref{cor:commutative flat is silting} is not easy and leads, in general, to an instance of Quillen's small object argument in the proof of~\cite[Theorem 6.3]{MS}, which is behind the proof of Proposition~\ref{prop:loc at proj silting}.

\section{From silting to flat: the commutative case} \label{sec:flat is silting}

Throughout this section, $A$ denotes a commutative, but \emph{not} necessarily noetherian ring. Our main aim, along with explaining specific features of the setting of commutative rings, is to prove a converse to Corollary \ref{cor:commutative flat is silting}, i.~e. that flat and silting ring epimorphisms coincide in the commutative setting. For this purpose, we need to recall some additional material on Gabriel filters on commutative rings.

\subsection{Thomason subsets and Gabriel filters}

Let $\Spec A$ be the Zariski spectrum of $A$. The Zariski closed sets in $\Spec A$ are precisely the sets of the form $V(I)=\{\p \in \Spec A\mid \p \supseteq I\}$ for some ideal $I$ of $A$. If $P \subseteq \Spec A$ is Zariski closed, its complement $\Spec A \setminus P$ is quasi-compact (in the subspace topology inherited from $\Spec A$) if and only if $P=V(I)$ for a finitely generated ideal $I$. The following concept is crucial in relation to flat epimorphisms of commutative rings.

\begin{definition}[{\cite{Th97}}]\label{def:Thomason}
A subset $P \subseteq \Spec A$ is said to be a \emph{Thomason subset} provided that it is a union of closed sets with quasi-compact complements. In other words, $P = \bigcup_{\lambda\in\Lambda} V(I_\lambda)$ for a collection $(I_\lambda\mid \lambda\in\Lambda)$ of finitely generated ideals of $A$.
\end{definition}

If $A$ is noetherian, the concept of Thomason set is, of course, equivalent to that of specialization closed set. Recall that $P$ is \emph{specialization closed} if it is a union of arbitrary Zariski closed sets.

\smallskip

Given a module $M$ and a full subcategory $\Ccal\subseteq\ModA$, 
we denote by $\Supp M=\{\p\in\Spec{A}\mid M\otimes_A A_\mathfrak{p}\not=0\}$
and $\Supp\Ccal=\bigcup_{M\in\Ccal}\Supp M$
the \emph{support} of $M$ and $\Ccal$, respectively. 
Every Thomason subset $P\subseteq\Spec{A}$ gives rise to a Gabriel filter on $A$ 
\[ \mathcal G_P=\{I\subseteq A\mid V(I)\subseteq P\}. \]
The corresponding hereditary torsion pair $(\mathcal T_P,\mathcal F_P)$ is given by the classes of \emph{$P$-torsion} and \emph{$P$-torsion-free} modules, respectively:
\begin{align*}
\mathcal T_P&=\{M\in \ModA\mid \Supp{M}\subseteq P\}, \\
\mathcal F_P&=\{M\in \ModA\mid \Hom{A}{A/I}{M}=0\text{ whenever }V(I)\subseteq P\}.
\end{align*}
The Giraud subcategory corresponding to $\Gcal_P$ is
\[\mathcal X_P
=\{M\in\ModA\mid\Hom{A}{A/I}{M}=\Ext{1}{A}{A/I}{M}=0\text{ whenever }V(I)\subseteq P\}.\]
Furthermore, $\Dcal_{\Gcal_P}$ coincides with the class of \emph{$P$-divisible} modules
\[ \Dcal_P=\{M\in\ModA\mid M\otimes_A A/I=0\text{ whenever }V(I)\subseteq P\}. \]

Notice that one can recover the set $P$ from the hereditary torsion pair $(\Tcal_P,\Fcal_P)$ by taking support:
$P=\Supp\Tcal_P.$ For details, we refer to \cite[\S2.2]{HSt} and, for the more classical case where $A$ is noetherian, also to \cite[\S VII.6]{St}.

\begin{theorem}[{\cite[Propositions 2.7 and 2.12]{HSt}, \cite[Theorem 4.7]{AH}}]\label{thm:Gabrielcorr}
For any commutative ring $A$, the assignments $P\longmapsto \Gcal_P$ and $P\longmapsto \Dcal_P$ define bijections between:
\begin{enumerate}
\item[(i)] Thomason subsets of $\Spec{A}$,
\item[(ii)] Gabriel filters of finite type on $A$,
\item[(iii)] silting classes in $\ModA$.
\end{enumerate}
\end{theorem}

\subsection{From silting to flat}

Now we shall focus on proving the converse of Corollary~\ref{cor:commutative flat is silting}. To start with, recall that if $\lambda\dd A\la B$ is a ring epimorphism with $A$ commutative, then $B$ is also commutative, and the induced map $\lambda^\flat\colon \Spec B \longrightarrow \Spec A$ is injective (see \cite[Corollary 1.2]{Si} and \cite[Proposition IV.1.4]{Laz}). This allows us to deduce the following criterion for the flatness of a ring epimorphism.

\begin{lemma} \label{lem:flatness locally}
Let $\lambda\dd A\la B$ be a ring epimorphism with $A$ commutative. Then $B$ is flat if and only if $\lambda_\p\colon A_\p \longrightarrow B_\p$ is flat for each $\p$ in $\im\lambda^\flat \subseteq \Spec A$.
\end{lemma}

\begin{proof}
If $\lambda$ is flat, so is clearly each $\lambda_\p$. Conversely, $B$ is a flat $A$-module if (and only if) $B_\q$ is a flat $A_\p$-module for each $\q\in\Spec B$ and $\p=\lambda^\flat(\q)$ by \cite[Theorem~7.1]{M}. It remains to observe that if $B_\p$ is flat over $A_\p$, then so is $B_\q \cong (B_\p)_\q$.
\end{proof}

We will also note that localization of a silting epimorphism is again a silting epimorphism.

\begin{lemma} \label{lem:loc silting}
Let $A$ be a commutative ring, $S\subseteq A$ be a multiplicative subset, and $T\in\ModA$ be a partial silting module with respect to a projective presentation $\sigma\colon P_1 \longrightarrow P_0$. Then $S^{-1}T\in\rmod{S^{-1}A}$ is a partial silting module with respect to the projective presentation $S^{-1}\sigma$. Moreover, if $\lambda\dd A\la B$ is the silting ring epimorphism given by $\sigma$, then $S^{-1}\lambda\dd S^{-1}A\la S^{-1}B$ is the silting ring epimorphism given by $S^{-1}\sigma$.
\end{lemma}

\begin{proof}
Let $T\in\ModA$ be partial silting with respect to $\sigma$. Since $T\in\Dcal_\sigma$ and $\Dcal_\sigma$ is definable by~\cite[Corollary~3.5 and Proposition~3.10]{AMV1}, we also have $S^{-1}T\in\Dcal_\sigma$. Moreover, for any $M\in\rmod{S^{-1}A}$, we clearly have that $\Hom{A}{\sigma}{M}$ is surjective (bijective) if and only if $\Hom{S^{-1}A}{S^{-1}\sigma}{M}$ is surjective (bijective), respectively. It follows that $\Dcal_{S^{-1}\sigma}$ is closed under coproducts in $\rmod{S^{-1}A}$ (so that $S^{-1}T$ is partial silting) and that $\Xcal_{S^{-1}\sigma}=\rmod{S^{-1}A}\cap\Xcal_\sigma\subseteq\ModA$ (so that $S^{-1}\lambda$ is the silting ring epimorphism given by $S^{-1}\sigma$).
\end{proof}

The main result of this section is as follows.

\begin{theorem} \label{thm:flat is silting}
Let $A$ be a commutative ring and $\lambda\dd A\la B$ be a ring homomorphism. Then $\lambda$ is a flat epimorphism if and only if $\lambda$ is a silting epimorphism.
\end{theorem}
\begin{proof}
If $\lambda\dd A\la B$ is a flat ring epimorphism, it is a silting epimorphism by Corollary~\ref{cor:commutative flat is silting}.

Suppose conversely that $\lambda$ is a silting epimorphism such that $\im\lambda_* = \Xcal_\sigma$ for a partial silting module $T$ with a projective presentation $\sigma\colon P_1 \longrightarrow P_0$. Then, according to Lemma~\ref{lem:loc silting}, $\lambda_\p\dd A_\p\la B_\p$ is a silting epimorphism such that $\im(\lambda_{\p})_*=\Xcal_{\sigma_\p}$ for any $\p\in\Spec A$. To prove that $B$ is flat over $A$, it suffices to prove that $B_\p$ is flat over $A_\p$ for every $\p\in\im\lambda^\flat\subseteq\Spec A$ according to Lemma~\ref{lem:flatness locally}. All in all, we may without loss of generality assume that $A$ is local and the maximal ideal $\m\subseteq A$ is in the image of $\lambda^\flat\colon \Spec B \longrightarrow \Spec A$. We will show that these assumptions imply not only that $\lambda$ is flat, but also that it is an isomorphism.

As $\lambda\otimes A/\m\colon A/\m \longrightarrow B/B\m$ is a ring epimorphism by~\cite[Lemma 4.1]{AMSTV} and $A/\m$ is a field, we infer from \cite[Corollary 1.2]{Si} that $\lambda\otimes_A A/\m$ is an isomorphism. In other words, $A/\m$ lies in $\Xcal_\sigma\subseteq\Dcal_\sigma$.
By Theorem~\ref{thm:finite type silting}, there exists a set $\Sigma$ of morphisms between finitely generated projective modules such that $\Dcal_\sigma = \Dcal_\Sigma$. Now, the map $\Hom{A}{\sigma}{A/\m}\cong A/\m\otimes_A\Hom{A}{\sigma}{A}$ must be surjective for each $\sigma$ in $\Sigma$, which by the Nakayama lemma implies that $\Hom{A}{\sigma}{A}$ is a split epimorphism and so $\sigma$ is a split monomorphism for each $\sigma\in\Sigma$. This in turn implies that $\Dcal_\sigma = \Dcal_\Sigma = \ModA$, so $\sigma$ is a split monomorphism and the partial silting $R$-module $T=\Coker\sigma$ is projective. Since $R$ is local, $T$ is either zero or a projective generator, and the latter is impossible since $\Xcal_\sigma\ne0$. In other words, $\sigma$ is an isomorphism and so is $\lambda\colon A \longrightarrow B$ is an isomorphism, as desired.
\end{proof}

\begin{remark}
As explained in Remark~\ref{rem:silting epi}(1),
the implication that silting epimorphisms are flat is much easier for commutative noetherian rings $A$.
In this case, $B$ is also commutative noetherian by~\cite[Corollaire IV.2.3]{Laz}.
\end{remark}

\bibliographystyle{abbrv}
\bibliography{flat-is-silting}

\end{document}